%% file: AccImplicitReg.tex
\documentclass[]{article}
\usepackage{graphicx} 

\usepackage{amsfonts}
\usepackage{amsmath}
\usepackage{amssymb}
\usepackage{amsthm}
\usepackage{euscript}
\usepackage{xcolor}
\usepackage[]{hyperref}
\usepackage{geometry}

\hypersetup{
citecolor={blue},
}

\usepackage[numbers,sort]{natbib}

\newtheorem{thm}{Theorem}
\newtheorem{lem}{Lemma}

\newtheorem{assump}{Assumption}

\input{short_commands.tex}

\title{Acceleration and Implicit Regularization in Gaussian Phase Retrieval}
\author{Tyler Maunu \qquad Martin Molina-Fructuoso\\ Brandeis University}
\date{November 2023}

\begin{document}
\maketitle

\begin{abstract}
  We study accelerated optimization methods in the Gaussian phase retrieval problem. In this setting, we prove that gradient methods with Polyak or Nesterov momentum have similar implicit regularization to gradient descent. This implicit regularization ensures that the algorithms remain in a nice region, where the cost function is strongly convex and smooth despite being nonconvex in general. This ensures that these accelerated methods achieve faster rates of convergence than gradient descent. Experimental evidence demonstrates that the accelerated methods converge faster than gradient descent in practice.
\end{abstract}

\section{Introduction}

While convex optimization is by now a mature field with an array of well-understood algorithms, the properties of nonconvex optimization algorithms have only recently begun to be illuminated. Understanding nonconvex optimization algorithms is of fundamental importance for progress in data science and machine learning since many optimization problems we wish to solve are nonconvex. Furthermore, it is important to understand how variants of these nonconvex optimization algorithms, such as stochastic or accelerated methods, perform in various settings. 

A prototypical nonconvex optimization problem arises when one considers the Burer-Monteiro factorization of programs over positive semidefinite (PSD) matrices \cite{burer2003nonlinear}. Here, for a PSD matrix $\bS\in \R^{n \times n}$ one uses the parametrization $\bS = \bU \bU^T$, and then applies a standard algorithm like gradient descent to optimize over $\bU$ rather than $\bS$. This factorization can be used in solving many problems, such as matrix completion, phase retrieval, covariance sketching, and more \cite{sanghavi2017local,li2019nonconvex,ma2020implicit}.
The Burer-Monteiro factorization has been shown to work well in practice despite it typically leading to nonconvex optimization problems. Furthermore, it represents a concrete case where we can understand the properties of nonconvex optimization algorithms. 

While many works have analyzed the optimization landscapes of factorized matrix problems, we draw our inspiration from the work of \citet{ma2020implicit}. Here, the authors study how gradient descent interacts with a random observation model on the specific problems of phase retrieval, matrix completion, and blind deconvolution. 
In this work, we focus on one of the specific problems of \citet{ma2020implicit} for clarity, although these results can be generalized. The \emph{phase retrieval} problem seeks a signal vector $\bx_\star \in \R^n$ from observations
\begin{equation}\label{eq:probs}
    y_i = |\ba_i^T \bx_\star|^2
\end{equation}
collected from a set of known sensing vectors $\ba_1, \dots, \ba_m$. In the general case, one considers sensing vectors and signals $\ba_i, \bx_\star \in \C^n$, but here we focus on the real case for simplicity of exposition. The Gaussian phase retrieval problem refers to the case where the sensing vectors $\ba_i$ are standard Gaussian \cite{candes2015phase}.

In effect, they show that gradient descent with proper initialization interacts well with the random measurement model by developing an analysis of its 
\emph{implicit regularization}.  Broadly speaking, implicit regularization refers to a bias in an optimization algorithm that only exists implicitly rather than explicitly.
In effect, the implicit regularization result of \citet{ma2020implicit} ensures that the optimization problem in question behaves much like a convex optimization problem despite the problem being nonconvex in general. 

Implicit regularization can be used to show that certain optimization methods perform well on nonconvex problems, and the type of regularization varies from setting to setting. In certain settings, it ensures that the algorithm finds a solution that minimizes some norm \cite{gunasekar2017implicit,li2018algorithmic,vaskevicius2019implicit,min2021explicit,zhao2022high} or yields outputs that are projections with respect to a Bregman divergence or minimum entropy solutions \cite{gunasekar2018characterizing,woodworth2020kernel,even2023s}.
Alternatively, as we mentioned in the case of \citet{ma2020implicit}, one can show that methods maintain some independence properties in random sensing models. Almost all implicit regularization results for nonconvex methods rely on finding a good starting point, either by spectral initialization or by some other cleverly chosen point.

While many results exist exploring the implicit biases of gradient descent and its stochastic counterpart, the implicit biases of accelerated algorithms are not well understood. It is essential to better understand their implicit biases since in practical nonconvex optimization such as training state-of-the-art deep learning methods, practitioners use some variant of momentum (such as in Adam \cite{kingma2014adam}).
Therefore, in this work, we set out to better understand the implicit regularization of gradient descent with momentum in the context of the specific nonconvex problem of phase retrieval. 

\subsection{Outline of Contributions}

Here, we summarize the main contributions of our work. These findings are also presented in Figure \ref{fig:accelex}: accelerated gradient descent (AGD) methods achieve faster convergence than gradient descent on the Gaussian phase retrieval problem in theory and practice.
\begin{enumerate}
    \item We develop the first implicit regularization analysis for two accelerated gradient methods on the Gaussian phase retrieval problem. These methods utilize Polyak's momentum \cite{polyak1964some} and Nesterov's momentum \cite{nesterov1983method}. Our analysis shows that with spectral initialization, these accelerated methods maintain incoherence and locality properties throughout their iterations. To accomplish this, we develop a novel leave-one-out analysis that works on pairs of iterates rather than the iterates themselves. To our knowledge, this is the first implicit regularization result for accelerated gradient methods on a matrix factorization problem.
    \item As a result of this analysis, we are able to show that accelerated gradient methods converge faster than gradient descent on the Gaussian phase retrieval problem. These are the first results of guaranteed convergence for accelerated algorithms on the phase retrieval problem, even though such methods were previously considered in works such as \cite{xu2018accelerated,wang2020quickly}.
    \item While past works experimentally verify the fast convergence of momentum methods, we also supplement our results to show that accelerated gradient methods achieve faster convergence than gradient descent in practice.
\end{enumerate}

\begin{figure}
    \centering
    \includegraphics[width = .9\columnwidth]{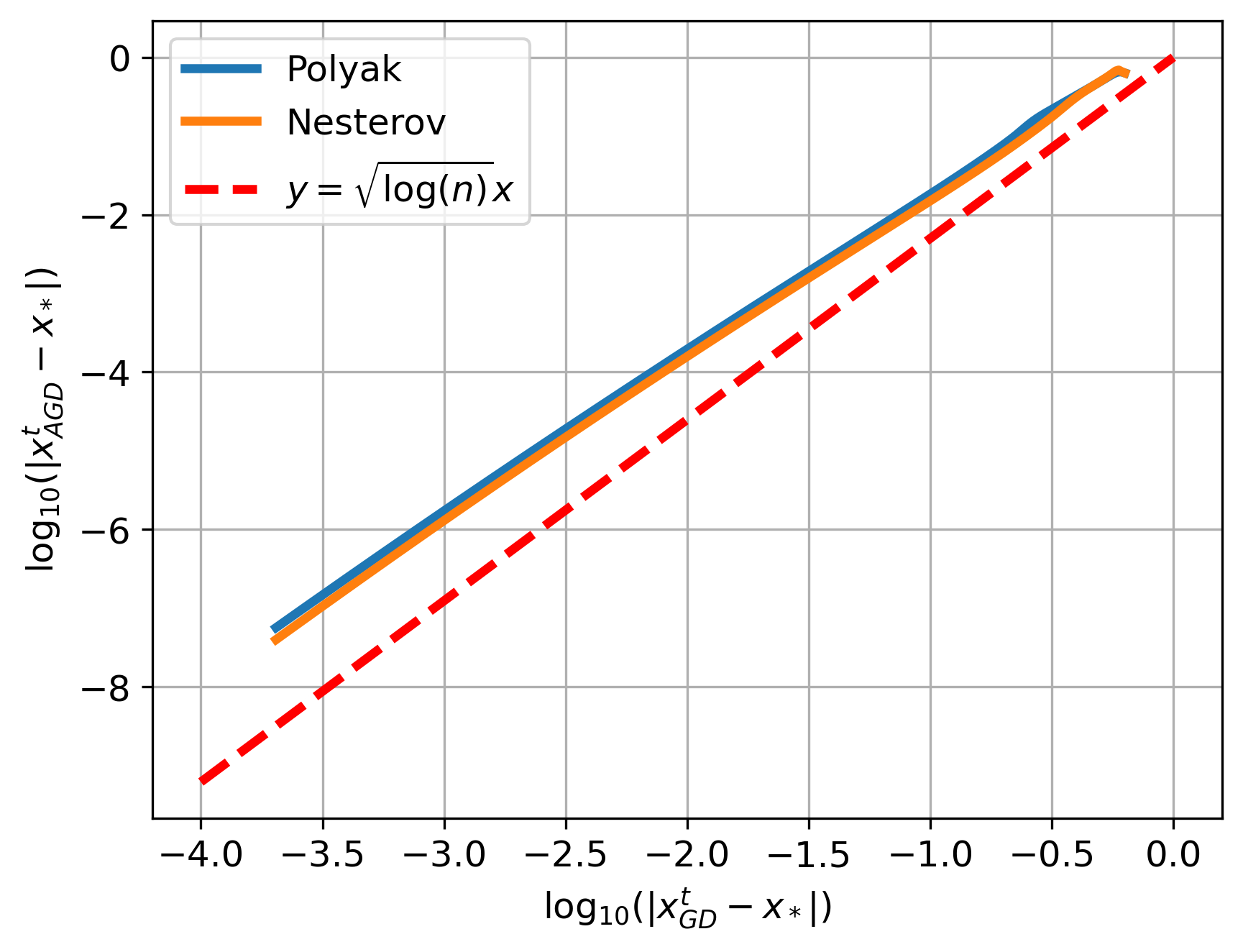}
    \caption{Accelerated gradient methods achieve faster convergence than gradient descent in both practice and in theory on the Gaussian phase retrieval problem. Here, Polyak and Nesterov momentum are used to accelerate gradient descent. The y-axis measures the log-error of the accelerated method, whereas the x-axis measures the log-error of the standard gradient method at each iteration. As we can see, after the same amount of iterations has passed for each method, accelerated gradient descent is approximately a factor of $\sqrt{\log(n)}$ closer to the ground truth on a log scale, which is supported by our theory.}
    \label{fig:accelex}
\end{figure}

\subsection{Review of Related Work}

Many techniques exist to solve the phase retrieval problem \eqref{eq:probs}, and so we review some of them here. 

Other variants of the phase retrieval problem consider different types of sensing vectors. These include holographic phase retrieval \cite{barmherzig2022towards}, ptychography \cite{thibault2008high}, near-field phase retrieval \cite{wang2020phase}, and far-field phase retrieval \cite{zhuang2022practical}. Here, we considered the problem of Gaussian phase retrieval \cite{candes2015phase}.

One popular technique for recovering the vector $\bx_*$ involves solving convex relaxations \cite{candes2013phaselift}. However, solving the resulting semidefinite programs is typically expensive. 

Many past works have studied first-order algorithms to recover $\bx^\star$ the observations \eqref{eq:probs}, and, as mentioned, these have shown efficient recovery.
\citet{candes2015phase} study the Wirtinger Flow algorithm, which has become a staple in the phase retrieval literature. This algorithm is gradient descent on the least squares objective that we consider. Here, the authors prove local linear convergence of the WF algorithm and show it needs $O(n \log(1/\epsilon)$ iterations to converge to an $\epsilon$-accurate solution. 
\citet{sun2018geometric} show that the objective has no local minimizers, all global minimizers are equal up to a phase shift, and all saddle points are strict saddles. To take advantage of this structure of the problem, they advocate the use of a trust-region algorithm.

Later, \citet{ma2020implicit} improved the analysis of gradient descent with an implicit regularization style analysis to $O(\log(n) \log(1/\epsilon))$. 
Even more recently, these results have recently been improved to $O(\log n + \log(1/\epsilon)$ by \cite{chen2019gradient,lee2023randomly}, whose results rely on random initialization. The work of \citet{chen2019gradient} focuses on a two-stage analysis of gradient descent, while the work of \citet{lee2023randomly} focuses on an alternating least squares approach.

Accelerated gradient methods for phase retrieval were considered in \cite{xu2018accelerated,wang2020quickly}.
Some preliminary results by \cite{wang2020quickly} showed that, at least in the population setting, heavy ball for phase retrieval converges to the benign region more quickly. However, proof of this fact for the finite sample setting has not yet been established.
Other work has considered higher-order methods, such as sketched Newton methods \cite{luo2023recursive}.

Also, some works have studied phase retrieval in alternative settings. Sparse phase retrieval is a common paradigm where the underlying signal of interest is assumed to be sparse \cite{eldar2014sparse,shechtman2014gespar,wang2017sparse}. The works \cite{wu2020continuous,wu2021hadamard} consider implicit regularization for sparse phase retrieval, where the underlying vector is assumed to be sparse. Others considered phase retrieval with generative priors, such as the work \cite{hand2018phase}. Finally, there are also works on robust phase retrieval, such as \cite{candes2013phaselift,duchi2019solving}.

Geometric approaches to solving the phase retrieval problem exist \cite{huang2017solving,hou2020fast,maunu2023matrix}, and some methods also consider alternative energies \cite{duchi2019solving,wang2018phase,maunu2023matrix}.

Finally, our work adds to the literature on implicit regularization of methods for nonconvex optimization. Much past work has focused on gradient descent, but some works have looked at accelerated methods as well \cite{pagliana2019implicit,wang2023implicit,ghosh2023implicit}. None of these results has considered the implicit regularization of accelerated regularization in matrix factorization problems, though.

\subsection{Notation}

We use bold lowercase letters to denote vectors and bold uppercase letters to denote matrices. When applied to vectors, $\|\cdot\|_2$ represents the Euclidean norm, whereas when applied to matrices, it represents the spectral norm.

\section{Background}

In this section, we give the necessary background to understand our contributions. First, we will outline simple approaches to solve the phase retrieval problem via first-order optimization algorithms. After this, we review related work on the use of accelerated first-order methods for the phase retrieval problem. Finally, we will review the general theory of first-order accelerated algorithms in convex optimization.

\subsection{Phase Retrieval}

The phase retrieval problem arises in settings such as X-ray crystallography \cite{candes2015phase}, and it also has applications to optical problems, in particular, medical optical imaging \cite{shechtman2015phase}, microscopy \cite{miao2008extending}, and astronomical imaging \cite{shechtman2015phase}.

While the practical phase retrieval problem considers various measurement operators $\cA$, such as Fourier observations, we consider a specific mathematical setting. This is also referred to as the \emph{Gaussian phase retrieval problem}.
\begin{assump}\label{assump:sensenormal}
    The sensing vectors $\ba_i$, $i=1, \dots, m$, are sampled i.i.d.~from a standard Gaussian distribution $\mathcal{N}(\boldsymbol{0},\bI)$.
\end{assump}
Other variants of the phase retrieval problem can be studied, where the sensing vectors take on other forms. See the examples in \cite{barmherzig2019holographic}. In ptychography, the observations are Fourier measurements of an objected modulated by a fixed illumination pattern focused over patches at a time \cite{thibault2008high}.

It is natural to consider recovering the underlying signal $\bx_\star$ by solving an optimization problem. For example, one can consider the least squares problem
\begin{equation}\label{eq:lsf}
    f(\bx) = \frac{1}{4m} \sum_{i=1}^m ((\ba_i^T \bx)^2 - y_i)^2.
\end{equation}
We note that this is equivalent to a Burer-Monteiro factorization \cite{burer2003nonlinear} on the low-rank matrix sensing problem \cite{fazel2008compressed}. In fact, this type of problem has a further structure in that the sensing matrices are rank one, which has been considered in \cite{chen2015exact,cai2015rop}.

Due to the ambiguity between $\pm \bx$ in the observations \eqref{eq:probs} as well as in the cost function \eqref{eq:lsf}, we can only hope to recover $\bx_*$ up to sign. Therefore, we adopt the following notion of distance as in \cite{ma2020implicit}.
\begin{equation}
    \dist(\bx, \bx_*) = \min \{ \|\bx - \bx_*\|, \|\bx + \bx^*\|\}.
\end{equation}

\citet{candes2015phase} introduced a gradient descent algorithm for the phase retrieval problem. 
The gradient of the functional \eqref{eq:lsf} is given by
\begin{equation}\label{eq:1der}
    \nabla f(\bx) = \frac{1}{m} \sum_{i=1}^m ((\ba_i^T \bx)^2 - y_i)\ba_i\ba_i^T \bx.
\end{equation}
The gradient descent algorithm follows the iterative updates
\begin{equation}
    \bx^{t+1} = \Big(\bI - \eta \frac{1}{m} \sum_{i=1}^m ((\ba_i^T \bx)^2 - y_i)\ba_i\ba_i^T\Big) \bx^t
\end{equation}
\citet{ma2020implicit} showed this algorithm to exhibit an implicit regularization property that leads to linear convergence,
\begin{equation}
\label{eq:GDRegResult}
    \|\bx^t-\bx_\star\| \leq C (1- \eta \|\bx_\star\|_2^2/ 2)^t \|\bx_\star\|_2
\end{equation}
for some absolute constant $C$.

\subsection{Accelerated Methods in Phase Retrieval}
\label{subsec:accmeth}

In this work, we consider accelerated first-order methods. These are not new and indeed have been considered in past works for phase retrieval, such as \cite{xu2018accelerated,ajayi2018provably,xiong2018convergence,xiong2020analytical,wang2020quickly}. In \cite{ajayi2018provably}, the authors study an adaptation of Nesterov's accelerated method for the matrix sensing problem that exhibits linear convergence up to an error that depends on the momentum parameter.
In \cite{wang2020quickly}, the authors study how a heavy ball performs with random initialization, although the proof is incorrect.
\citet{xiong2018convergence,xiong2020analytical} show convergence of \eqref{eq:HB} and \eqref{eq:FG} using techniques from control theory. However, they do not demonstrate acceleration over standard gradient methods and require a generic regularity condition. In \cite{xu2018accelerated}, the authors demonstrate the advantages of accelerated Wirtinger flow but offer no proof of convergence.
None of these past works show convergence to the ground truth signal, and furthermore, no past works consider these accelerated gradient methods from the lens of implicit regularization. Therefore, our work answers the important outstanding question in the literature on the convergence of these momentum methods. 

The two accelerated gradient methods considered in this paper are Polyak's Heavy Ball \cite{polyak1964some} and Nesterov's Accelerated Gradient \cite{nesterov1983method} methods. The iteration that defines gradient descent with Polyak momentum is
\begin{align}\label{eq:HB}\tag{\sf{P}}
    \bx^{t+1} &= \bx^t - \eta \nabla f(\bx^t) + \beta(\bx^t - \bx^{t-1})
\end{align}
where $\bx^1 = \bx^0$. On the other hand, gradient descent with Nesterov momentum takes the form
\begin{align}\label{eq:FG}\tag{\sf{N}}
    \bx^{t+1} &= \bx^{t} - \eta \nabla f(\bx^{t} +\beta(\bx^t - \bx^{t-1})) + \beta(\bx^{t} - \bx^{t-1})
\end{align}
where again $\bx^1=\bx_0$.

We note that the primary difference between the two is that Nesterov momentum takes the gradient at an extrapolated point, whereas Polyak momentum uses the gradient at the current iterate. As we point out in our analysis, controlling the Nesterov momentum's implicit regularization is more challenging due to this extrapolation.

In some convex settings, the convergence of these accelerated methods is well understood. For a twice continuously differentiable function $f$, we say that $f$ is $\mu$-strongly convex and $L$-smooth if 
\begin{equation}
    \mu \bI \preceq \nabla^2 f(\bx) \preceq L \bI,
\end{equation}
for all $\bx$. The condition number is the ratio $\kappa = L / \mu$. 
Under the assumption of strong convexity and smoothness, the convergence of \eqref{eq:HB} and \eqref{eq:FG} can be succinctly stated in the following theorem.
\begin{thm}[\cite{polyak1964some,bubeck2015convex}]\label{thm:convcompare}
    Suppose that we run gradient descent, \eqref{eq:HB} or \eqref{eq:FG} on a $\mu$-strongly convex and $L$-smooth function $f$. Then, gradient descent with step size $\eta = 1/L$ achieves the convergence bound
    \begin{equation}
        \|\bx_t - \bx_*\|_2 \lesssim \exp(- t/{\kappa}) \|\bx_t - \bx_0\|_2.
    \end{equation}
    On the other hand, \eqref{eq:HB} with step size $\eta=\frac{4}{(\sqrt{\mu}+\sqrt{L})^2}$ and momentum parameter $\beta = \frac{\sqrt{\kappa}-1}{\sqrt{\kappa}+1}$ and \eqref{eq:FG} with step size $\eta = 1/L$ and momentum parameter $\beta = \frac{\sqrt{\kappa}-1}{\sqrt{\kappa}+1}$ achieve the convergence bound
    \begin{equation}
        \|\bx_t - \bx_*\|_2 \lesssim \exp(-c_{1}t/\sqrt{\kappa}) \|\bx_t - \bx_0\|_2,
    \end{equation}
    Here, the constant $c_1 = 2$ for \eqref{eq:HB} and $c_1 = 1$ for \eqref{eq:FG}.
\end{thm}
As we can see, the accelerated methods achieve a faster rate, where the improvement is by a factor of $\sqrt{\kappa}$. This means that, to reach an error $\epsilon$, GD needs $O(\kappa \log(1/\epsilon))$ iterations while the accelerated methods need $O(\sqrt{\kappa}\log(1/\epsilon))$ iterations.
Notice that one needs to select the parameters for the momentum methods carefully. However, as we will see in the next section, in the Gaussian phase retrieval problem, we have bounds for $\mu$ and $L$ that allow us to set the parameters in these methods.

\section{Convergence for Accelerated Gradient Methods}

In this section, we give our main theorems proving convergence of \eqref{eq:HB} and \eqref{eq:FG} on the Gaussian phase retrieval problem. First, in Section \ref{subsec:impreg}, we review the ideas of implicit regularization from \cite{ma2020implicit} used to prove the convergence of gradient descent on this problem. Then, in Section \ref{subsec:mainthm}, we give our main convergence theorems and a sketch of their proofs. Finally, Section \ref{subsec:disc} discusses these results and compares them with those for gradient descent.

\subsection{Implicit Regularization}
\label{subsec:impreg}

\citet{ma2020implicit} give implicit regularization results for low-rank matrix recovery problems. The main idea is to ensure that the iterates of gradient descent remain in a good region, which the authors call the \textit{region of incoherence and contraction (RIC)}. In this region, one can show that the cost function \eqref{eq:lsf} is strongly convex and smooth.

To see how this might be the case, consider the Hessian of the cost function,
\begin{equation}\label{eq:2der}
    \nabla^2 f(\bx) = \frac{1}{m} \sum_{i=1}^m (3(\ba_i^T \bx)^2 - y_i)\ba_i\ba_i^T.
\end{equation}
One can show that $ \bI \preceq \E_{\ba_i} \nabla^2 f(\bx) \preceq 10 \bI$ for $\bx$ in a small ball around $\bx_*$,
\begin{equation}\label{eq:loc}\tag{\sf{LOC}}
    \|\bx^t - \bx_\star\|_2 \leq 2 C_1 \|\bx_\star\|_2
\end{equation}
This is a locality property, which says the iterate needs to be sufficiently close to optimal for the cost function to be strongly convex and smooth. Thus, we might expect the function to be smooth and strongly convex for large samples. However, the concentration of the Hessian requires a suboptimal number of samples $m=\Omega(n^2)$ samples, while $m = \Omega(n)$ samples is sufficient to identify $\bx_*$ \cite{chen2015exact}. 

To transfer these ideas to the small sample setting, one can show that the Hessian of $f$ is sufficiently well-behaved, provided that $\bx$ obeys an additional condition besides locality. The extra condition that defines the RIC besides \eqref{eq:loc} is the incoherence condition,
\begin{equation}\label{eq:inc}\tag{\sf{INC}}
    \max_{i} |\ba_i^T(\bx - \bx_\star)| \leq C_2 \log(n) \| \bx_* \|.
\end{equation}
This condition ensures that the vector $\bx-\bx_*$ is not too aligned with any specific vector $\ba_i$.

Under the two assumptions \eqref{eq:loc} and \eqref{eq:inc}, Lemma 1 of \cite{ma2020implicit} states that with probability at least $1-\cO (mn^{-10})$,
\begin{equation}
\label{eq:MaLOCINC}
    \frac{1}{2} \bI \preceq \nabla^2 f(\bx) \preceq O(\log(n)) \bI.
\end{equation}
The main idea is that the additional restriction of incoherence allows one to apply concentration bounds that hold for sample sizes $m=\Omega(n\log n)$ rather than $m=\Omega(n^2)$. 

With this lemma in hand, one might expect to be able to show that \eqref{eq:HB} and \eqref{eq:FG} achieve faster convergence due to Theorem \ref{thm:convcompare}. However, there is an outstanding issue: this theorem assumes that the optimization landscape is globally strongly convex and smooth. However, this is not the case for the phase retrieval function, since it is only strongly convex and smooth inside the RIC.

To find an initial point in the RIC, we use a spectral initialization \cite{ma2020implicit}. Indeed, setting
$$\bx^0=\sqrt{\lambda_1(\bY)/3} \widetilde{\bx}^0,$$ 
where $\lambda_1(\bY)$ and $\widetilde{\bx}^0$ are the leading eigenvalue and eigenvector of $\bY=\frac{1}{m} \sum_{i=1}^{m} y_i \ba_i \ba_i^T $, guarantees that inequalities \eqref{eq:loc} and \eqref{eq:inc} are satisfied for $\bx^0$.

\subsection{Convergence of (\ref{eq:HB}) and (\ref{eq:FG})}
\label{subsec:mainthm}

The main theorem in this work is the same as the main theorem of \cite{ma2020implicit} for phase retrieval with the notable change that the convergence improves from $\kappa$ to $\sqrt{\kappa}$, as in Theorem \ref{thm:convcompare}. In practical terms, with spectral initialization, \eqref{eq:HB} and \eqref{eq:FG} need $O(\sqrt{\log n} \log(1/\epsilon))$ iterations rather than the $O(\log n \log(1/\epsilon))$ to reach an error $\epsilon$. As we demonstrate in experiments, this is backed up by real scenarios, where we see accelerated method demonstrating faster convergence.
\begin{thm}[Convergence of \eqref{eq:HB} and \eqref{eq:FG}]\label{thm:HB}
Suppose that $\bx_*$ is a fixed vector with $\|\bx_*\|=1$ and $\ba_j$ follow Assumption \ref{assump:sensenormal}. Provided that $0 \leq \eta \lesssim \frac{1}{c\log n \|\bx_0\|^2}$, $m \gtrsim n \log n$,  $\beta=\frac{\sqrt{c\log(n)}-\sqrt{1/2}}{\sqrt{c\log(n)}+\sqrt{1/2}}$ for some sufficiently large constant $c$, then with probability at least $1 - O(mn^{-9})$, \eqref{eq:HB} and \eqref{eq:FG} with spectral initialization achieve the following contraction and incoherence:
\begin{equation}
    \dist(\bx^t, \bx_*) \leq \epsilon \left(1 - \sqrt{\eta} \|\bx_*\|_2^2/2 \right)^t \| \bx_*\|_2
\end{equation}
\begin{equation}
    \max_j |\ba_j^T(\bx^t - \bx_*)| \leq c_2 \sqrt{\log n} \|\bx_*\|_2.
\end{equation}
\end{thm}

To fix ideas, we present the main ideas in the proof for Polyak momentum. The proof for Nesterov momentum needs similar steps and estimations with more involved expressions. The fundamental ingredient of the proof of the rate of our algorithm is the fact that, for a strongly convex and smooth function, two consecutive iterations of \eqref{eq:HB} with the parameters of Theorem \ref{thm:HB} achieve the contraction 
     \begin{equation}
     \label{eq:HBcont}
        \left\| \begin{bmatrix}
            \bx^{t+1} - \bx_* \\
            \bx^t -\bx_*
        \end{bmatrix} \right\|_2 \leq \left(\frac{\sqrt{L}-\sqrt{\mu}}{\sqrt{L}+\sqrt{\mu}} \right) 
        \left\|
        \begin{bmatrix}
            \bx^t-\bx_* \\
            \bx^{t-1} -\bx_*
        \end{bmatrix}
        \right\|_2.
    \end{equation}

The repeated application of this relation, i.e., 
    \begin{equation}
        \left\| \begin{bmatrix}
            \bx^{t+1} - \bx_* \\
            \bx^t -\bx_*
        \end{bmatrix} \right\|_2 \leq \left(\frac{\sqrt{L}-\sqrt{\mu}}{\sqrt{L}-\sqrt{\mu}} \right)^t
        \left\|
        \begin{bmatrix}
            \bx^1-\bx_* \\
            \bx^{0} -\bx_*
        \end{bmatrix}
        \right\|_2
    \end{equation}

would be enough to show the algorithm's convergence if $f$ was strongly convex and smooth. However, because $f$ is nonconvex in general, to use this reasoning directly, we need to make sure that all the iterations of gradient descent remain in the RIC, i.e., that they satisfy \eqref{eq:loc} and \eqref{eq:inc}.

While the detailed proof of Theorem \ref{thm:HB} is given in the supplementary material, we sketch the induction we use here.

\begin{enumerate}
    \item First, we notice that for points in the RIC (namely close enough to $\bx_*$ \eqref{eq:loc} and in the incoherence region \eqref{eq:inc}), the sampled function \eqref{eq:lsf} is strongly convex and smooth with high probability.
    \item Second, we prove a contraction-like result for iterations $\bx_t$ in the RIC of the form
    \begin{equation}
        \left\| \begin{bmatrix}
            \bx^{t+1} - \bx_* \\
            \bx^t -\bx_*
        \end{bmatrix} \right\|_2 \leq M
        \left\|
        \begin{bmatrix}
            \bx^t-\bx_* \\
            \bx^{t-1} -\bx_*
        \end{bmatrix}
        \right\|_2,
    \end{equation}
    where $M$ depends only on the constants of strong convexity and smoothness of $f$. This result follows from the facts that the segment $[\bx_*,\bx^t]$ is in the RIC for \eqref{eq:HB} and that $f$ can be considered strongly convex and smooth in the RIC \eqref{eq:MaLOCINC}.  In a similar way, we can show that for\eqref{eq:FG}, the segment $[\bx_*,\bx^t + \beta(\bx^t-\bx^{t-1})]$ is in the RIC. 
    \item We then proceed by induction to show that all iterations for a starting point $\bx_0$ in the RIC remain in the RIC with high probability using a leave-one-out argument. 
\end{enumerate}

Our proof strategy is based on a leave-one-out approach, where a new function is constructed
\begin{equation}\label{eq:costloo}
    f^{(\ell)}(\bx) = \frac{1}{4m} \sum_{i: i \neq \ell} ((\ba_i^T \bx)^2 - y_i)^2,
\end{equation}
for $\ell = 1, \dots, m$. The leave one out sequences $\bx^{t, (\ell)}$ are defined by running one of the optimization methods \eqref{eq:HB} or \eqref{eq:FG} on \eqref{eq:costloo}. The following lemma ensures that the accelerated methods remain close to the leave-one-out sequences throughout their iterations.
\begin{lem}[Leave-one-out proximity]\label{}
 Suppose that $\ba_j \overset{i.i.d.}{\sim}N(\bzero, \bI) $. Then, for the sequences generated by \eqref{eq:HB} or \eqref{eq:1der}, with probability at least $1-\cO(mn^{-10})$,
$$\max_{1 \leq \ell \leq m} \left\| \begin{bmatrix}
    \bx^{t+1}-\bx^{t+1,(\ell)} \\
    \bx^t -\bx^{t,(\ell)}
\end{bmatrix}\right\|_2\leq C_3 \sqrt{\frac{\log(n)}{n}}.$$
\end{lem}
In the past analysis of gradient descent, proximity to leave-one-out sequences was essential for proving that the iterates remain in the RIC. Therefore, our analysis can be seen as an extension of these results to the accelerated case, where we show that the accelerated gradient sequences also remain close to the leave-one-out sequences.

\subsection{Discussion}
\label{subsec:disc}

It is important to point out that we only actually need $O(\log n)$ iterations to reach a region where \eqref{eq:loc} implies \eqref{eq:inc}. From that point on, the rates of convergence provided by the methods we studied are not probabilistic. This fact is also used in \cite{ma2020implicit}, although they use a conservative bound on the number of iterations to achieve this fact.

One of the main challenges in the proof of implicit regularization and convergence of the accelerated gradient methods is that these methods do not generally define monotonic sequences for common Lyapunov functions. For example, we do not expect the sequences $(\|\bx^t - \bx_*\|)_{t \in \N}$ or $(f(\bx^t) - f(\bx_*))_{t \in \N}$ to be monotonic. This then makes it hard to guarantee that the sequences of iterates $(\bx^t)$ remain within the RIC. 
Fortunately, we are able to prove an analogous implicit regularization result to that observed for gradient descent in \cite{ma2020implicit} by looking at sequences that consider pairs of consecutive iterates. That is, we are able to prove a contraction result, as well as leave-one-out proximity for the sequence $\left( \begin{bmatrix}
    \bx^{t+1} \\ \bx^t
\end{bmatrix}\right)_{t \in \N}$.

Since the bounds require that pairs of iterates maintain properties rather than individual iterates, the constants in our estimates are worse than those required for GD by a constant absolute factor. This essentially means that our sample complexity is a constant factor worse than that required for GD. On the other hand, our convergence rate is arbitrarily better than that of GD, since we get a better convergence rate by a factor of $\sqrt{\log n}$.  

There are a few limitations of our methodology. First, current results on random initialization and alternating projections separate the dimension from from the $\log(1/\epsilon)$ in the rate, while our rate is $\sqrt{\log n} \log(1/\epsilon)$. Thus, while our results are better than GD with spectral initialization \cite{ma2020implicit}, they are not better than GD with random initialization in the random model \cite{chen2019gradient}. Proving convergence of accelerated methods with random initialization remains an open question \cite{wang2020quickly}. On the other hand, our work opens the door to accelerating in the random setting by showing how to apply leave-one-out analysis to accelerated gradient methods.

\section{Experiments}

While our results are primarily theoretical in nature, we give a few simulations showing that accelerated gradient methods are well-behaved, and actually perform better than GD on simulated phase retrieval tasks.

Our first experiment is displayed in Figure \ref{fig:comp1}. Here, we test GD against the two momentum based methods \eqref{eq:HB} and \eqref{eq:FG}. Here, the sensing vectors satisfy Assumption \ref{assump:sensenormal}, $\bx_*$ is drawn uniformly from $S^{n-1}$, $\eta = 0.05/(\log n)$, and $\beta = \frac{\sqrt{\log n} - \sqrt{2}}{\sqrt{\log n} + \sqrt{2}}$. We vary $n=10, 50, 100$ and $m=200, 500, 1000$. As we can see, where gradient descent converges, the accelerated methods converge faster.

\begin{figure}[ht]
    \centering
    \includegraphics[width=.6\columnwidth]{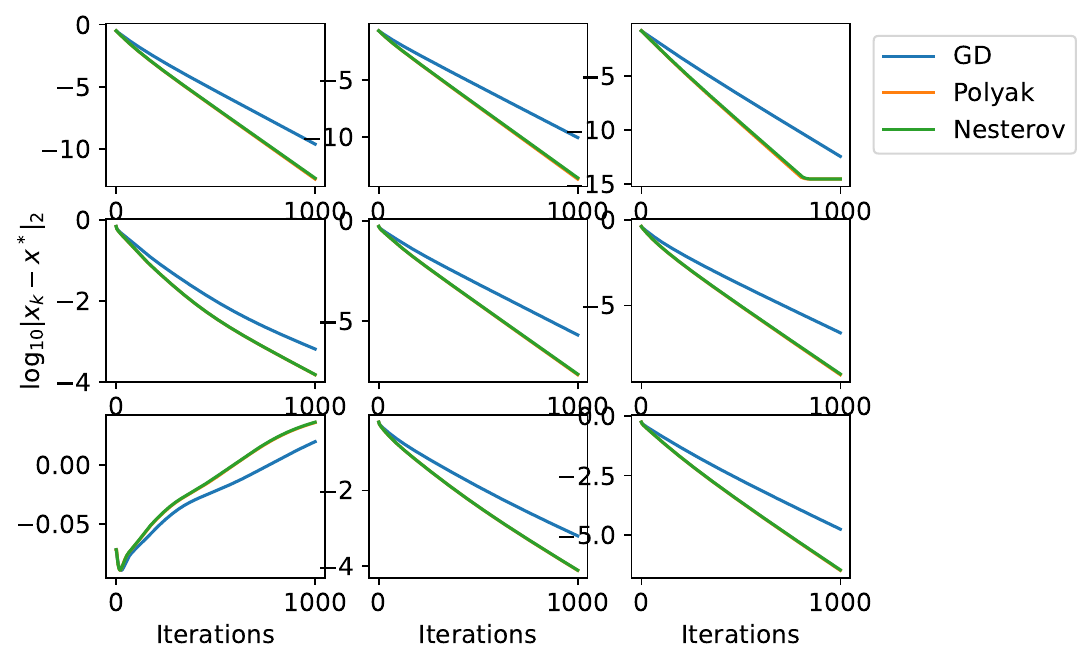}
    \caption{Experiment demonstrating the faster convergence of accelerated methods when compared to GD. Across the rows we vary $m=200, 500, 1000$, and down the columns we vary $n=10, 50, 100$. As we can see, \eqref{eq:HB} and \eqref{eq:FG} have identical performance.}
    \label{fig:comp1}
\end{figure}

In our second experiment, we test GD against the two momentum-based methods \eqref{eq:HB} and \eqref{eq:FG} in the same setup as before, but now the initial vector is randomly chosen from $S^{n-1}$. The results are displayed in Figure \ref{fig:comp12}. As we can see, the accelerated methods still converge faster than GD when it converges. Therefore, we expect that our theoretical results can be extended to this setting as well.

\begin{figure}[ht]
    \centering
    \includegraphics[width=.6\columnwidth]{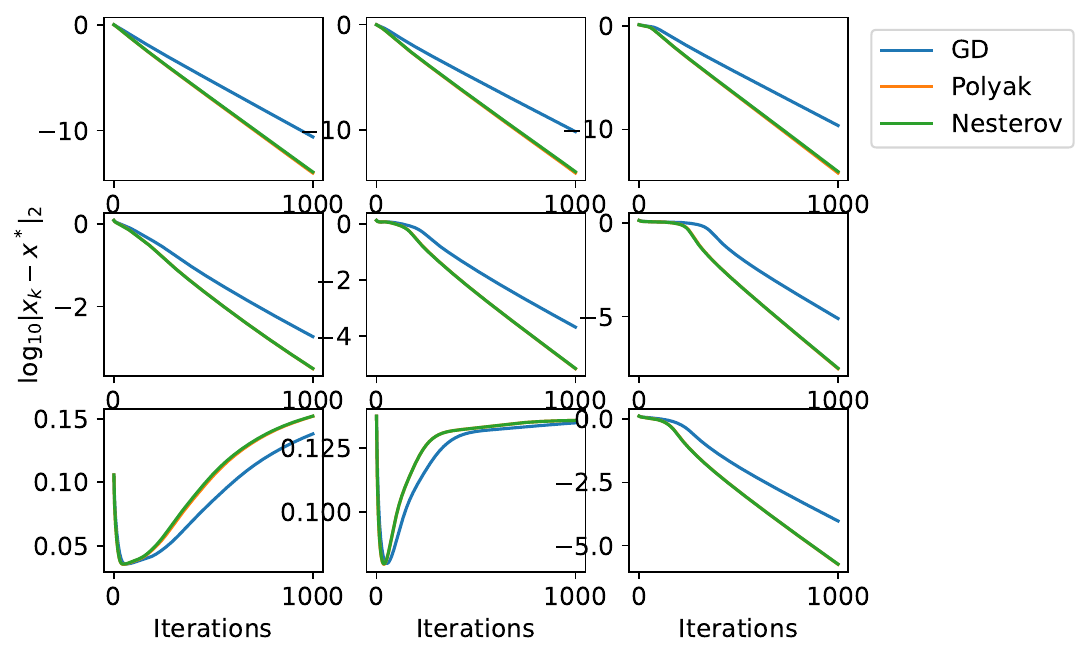}
    \caption{Experiment demonstrating the faster convergence of accelerated methods when compared to GD with random initialization. Across the rows we vary $m=200, 500, 1000$, and down the columns we vary $n=10, 50, 100$. As we can see, \eqref{eq:HB} and \eqref{eq:FG} have identical performance.}
    \label{fig:comp12}
\end{figure}

In our third experiment in Figure \ref{fig:slopes}, we repeat the experiment in Figure \ref{fig:accelex} across a range of dimensions. For each dimension, we compute the slope of the line given by plotting $\log \|\bx^t_{AGD}-\bx_*\|$ against $\log \|\bx^t_{GD}-\bx_*\|$. This slope can be interpreted as the speedup of the accelerated method. With the spectral initialization, we expect the speedup to be proportional to $\sqrt{\log n}$, which we observe in practice. 
\begin{figure}[ht]
    \centering
    \includegraphics[width=0.5\columnwidth, trim=0 0 0 0, clip]{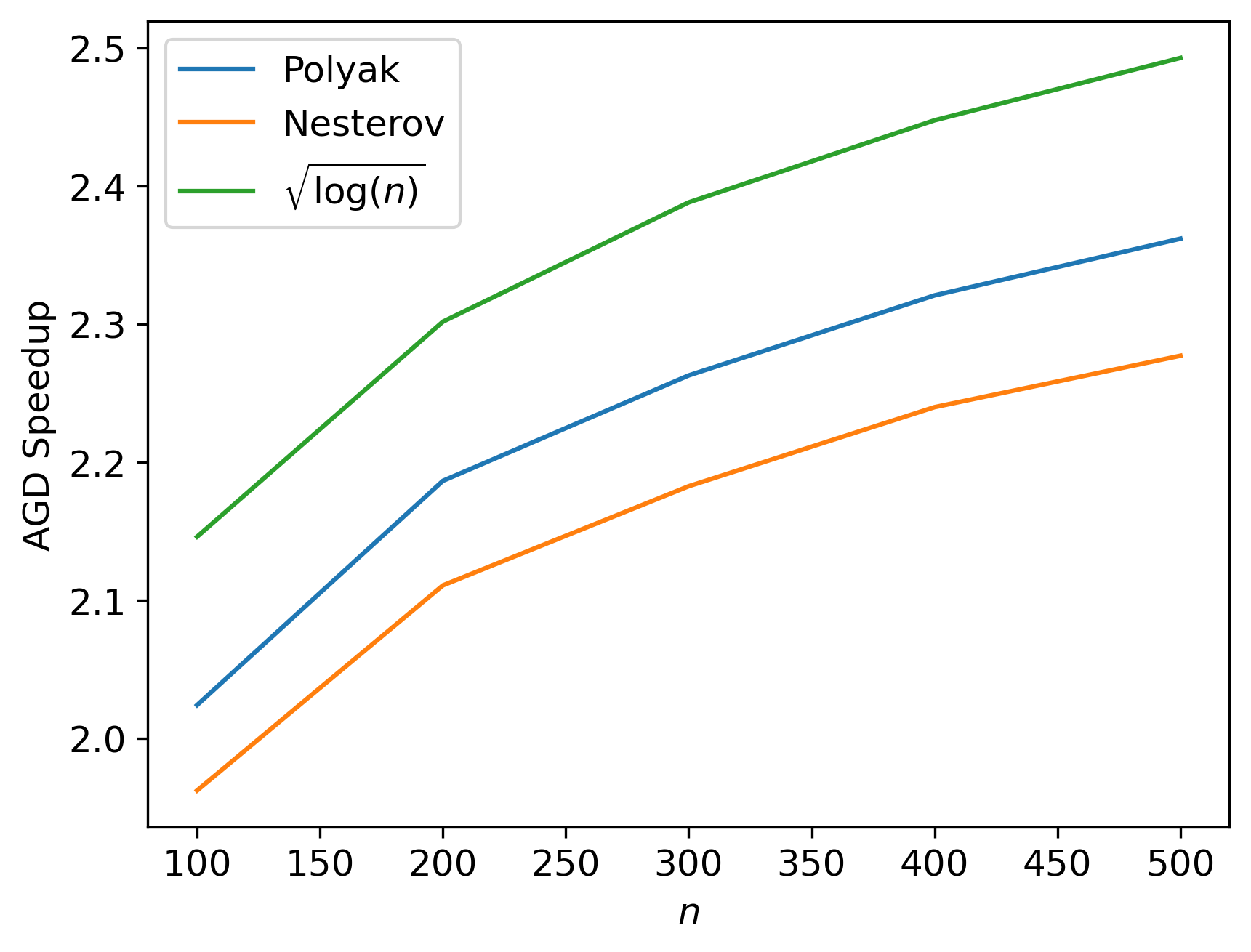}
    \caption{Plot of speedup of accelerated gradient descent versus momentum. Here, the y-axis is the slope of the linear fit of $\log \|\bx^t_{AGD}-\bx_*\|$ versus $\log \|\bx^t_{GD}-\bx_*\|$, and the x-axis is $n$. As we can see, these are well-approximated by the curve $\sqrt{\log n}$.}
    \label{fig:slopes}
\end{figure}

Finally, we compare gradient descent to the accelerated methods on the coded diffraction experiment of Section IV of \cite{candes2015phase}. Here, the vector $\bx_*$ is the Stanford main quad image, and $\ba_j$ are Fourier measurements with random modulation patterns; see the supplementary material or Section IV or \cite{candes2015phase} for details. As we can see, the accelerated methods recover the image much faster than Wirtinger Flow (WF), which is just gradient descent.

\begin{figure}[ht]
    \centering
    \includegraphics[width = 0.5\columnwidth]{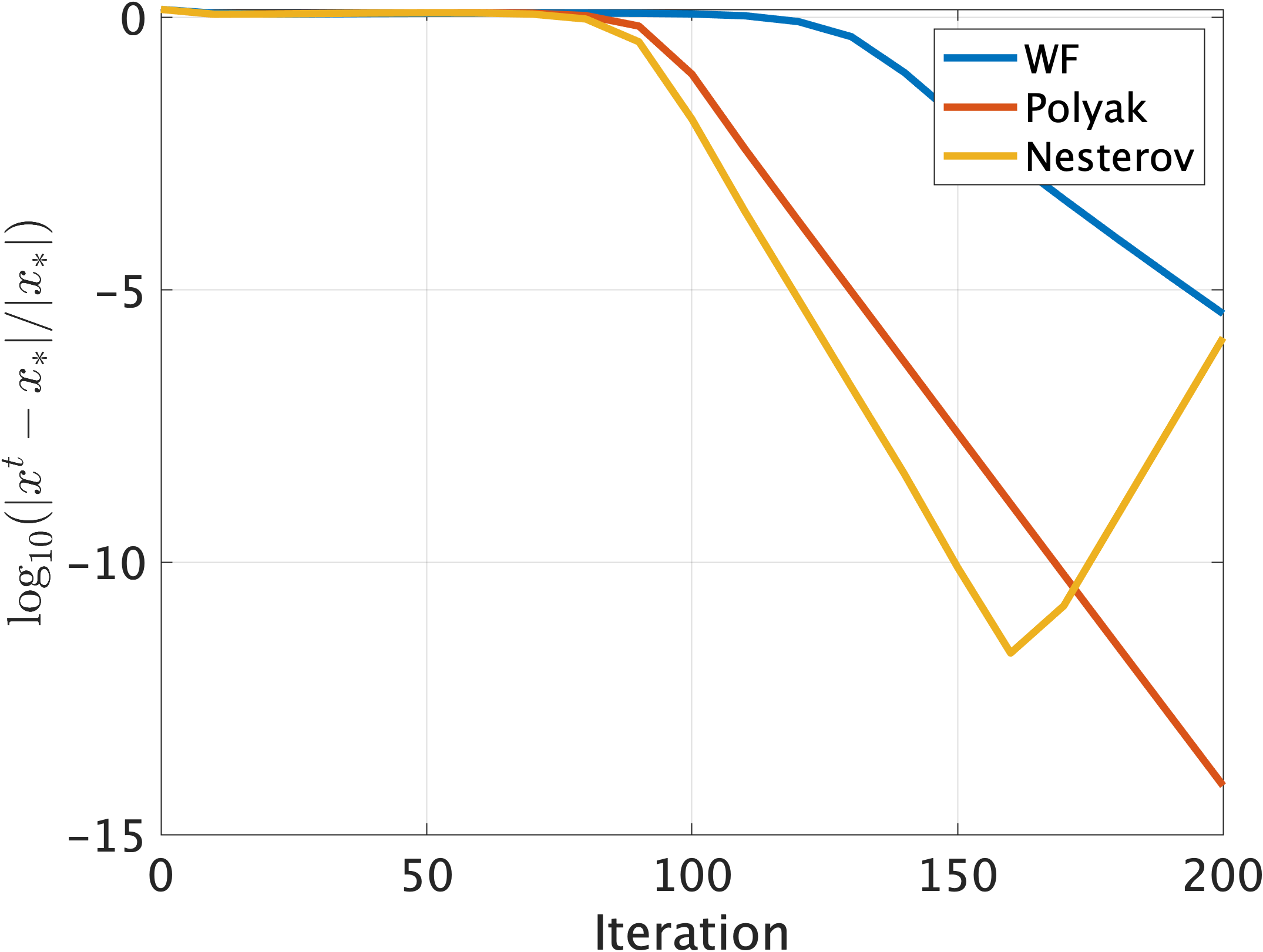}
    \caption{Errors versus iteration for Wirtinger flow as well as Nesterov/Polyak accelerated Wirtinger flow. All methods have the same per-iteration complexity. As we see, the accelerated methods recover the true image much faster.}
    \label{fig:codediffex}
\end{figure}
\begin{figure}[ht]
    \centering
    \includegraphics[width = 0.9\columnwidth]{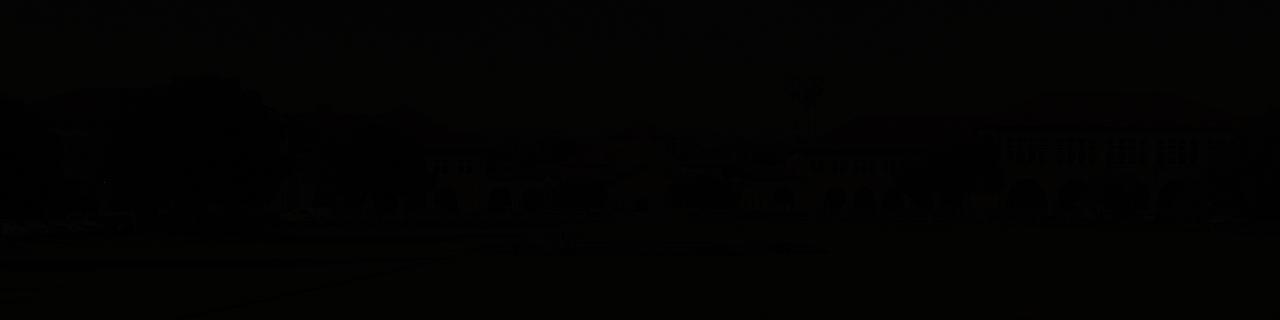}
    \includegraphics[width = 0.9\columnwidth]{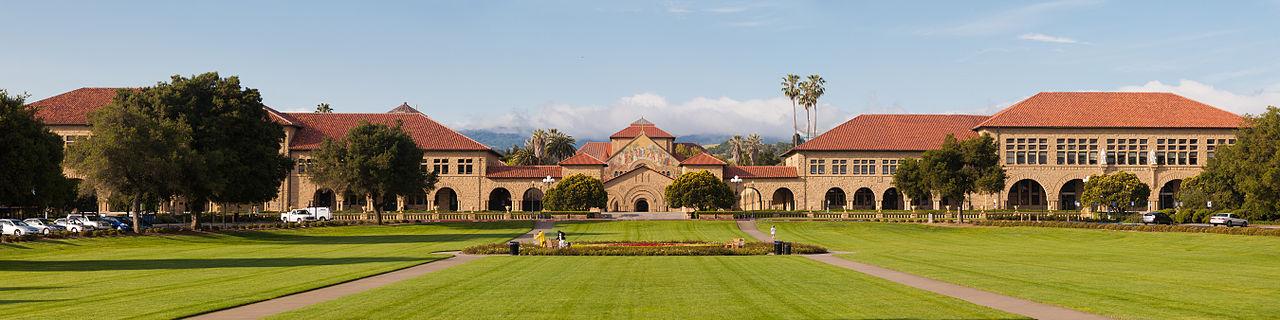}
    \includegraphics[width = 0.9\columnwidth]{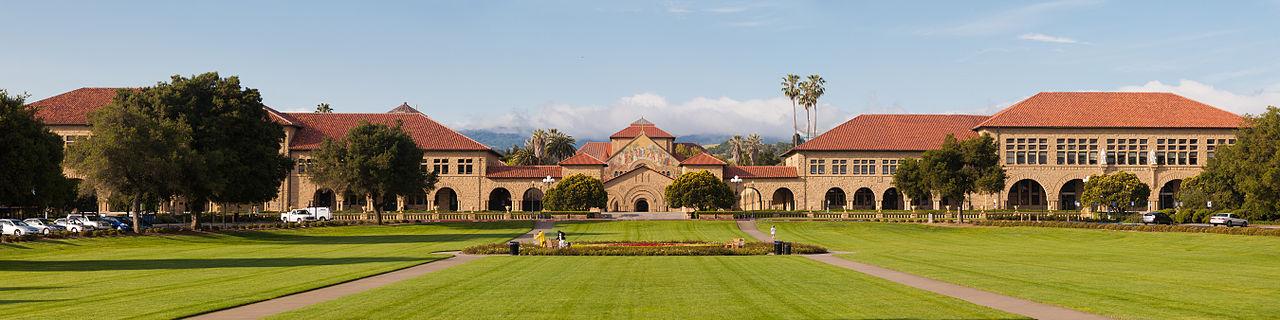}
    \caption{Recovered images after 140 iterations for the accelerated as well as Wirtinger flow algorithm on the coded diffraction test. As we can see, the accelerated methods have already recovered an accurate image before the gradient method even begins to converge. Note that the WF image is black because most pixels are much smaller in value than the maximum.}
    \label{fig:codediffex2}
\end{figure}

\section{Conclusions}

This work presents the first analysis of implicit regularization of accelerated gradient methods on a matrix factorization problem. We consider the problem of Gaussian phase retrieval, where the measurement vectors follow a standard Gaussian distribution. In this case, we are able to show that accelerated gradient iterates remain well-behaved in the sense that they remain within the region of incoherence and contraction. We believe that our results can be generalized to other settings, such as matrix completion and blind deconvolution. 

Many interesting future directions can be pursued. For example, one can consider extending the results to show that accelerated methods converge more quickly than GD for random initialization. Furthermore, one can consider accelerated methods for sparse phase retrieval or generative phase retrieval. This could be done by applying ideas from nonsmooth optimization \cite{beck2009fast} or accelerating smooth formulations with other implicit regularization \cite{wu2021hadamard}.

\section{Acknowledgements}

T. Maunu acknowledges the support of the NSF under
Award No. 2305315.

\bibliography{refs}
\bibliographystyle{plainnat}

\newpage

\onecolumn

\appendix
\section{Proof of Main Theorems}

Recall that, by Assumption \ref{assump:sensenormal}, the sensing vectors $\{\ba_i\}$, $i=1,\ldots,m$, are sampled i.i.d from a normal distribution $\cN(\boldsymbol{0},\bI)$. In the proofs in this section, we use the following results that come from standard concentration inequalities for normal distributions, see \cite[Appendix A]{ma2020implicit}:
\begin{enumerate}
    \item With probability at least $1-\cO(me^{-3n/2})$:
    \begin{equation}
    \label{eq:concnorm}
        \max_{1 \leq i \leq m} \| \ba_i \|_2 \leq \sqrt{6n}
    \end{equation}
    \item With probability at least $1-\cO(mn^{-10})$, for any $\by \in \mathbb{R}^n$:
    \begin{equation}
    \label{eq:concproj}
        \max_{1\leq i \leq m} |\ba_i^T \by| \leq 5 \sqrt{\log(n)} \| \by \|_2
    \end{equation}
\end{enumerate}

\subsection{Proof of Theorem \ref{thm:HB} for Polyak's Heavy Ball method}
\label{app:HBpf}

\begin{proof}[Proof of Theorem \ref{thm:HB} for Polyak's Heavy Ball Method \eqref{eq:HB}]
By repeated application of Lemma \ref{lem:induction}, for any $t \leq n$, we have that $\bx^t$ is in the RIC with probability at least $(1-\cO(mn^{-10}))^{n} \geq 1- \cO(mn^{-9})$. Moreover, by Lemma \ref{lem:contraction} we also have that
\begin{align}
    \left\|
    \begin{bmatrix}
    \bx^{t+1} - \bx_* \\
    \bx^t - \bx_*
    \end{bmatrix}
\right\|_2 \leq & \left(\frac{\sqrt{L}-\sqrt{\mu}}{\sqrt{L}+\sqrt{\mu}} \right)^t 
\left\|
    \begin{bmatrix}
    \bx^{1} - \bx_* \\
    \bx^0 - \bx_*
    \end{bmatrix}
\right\|_2 \nonumber \\
    \leq & \left(1-\frac{2\sqrt{\mu}}{\sqrt{L}+\sqrt{\mu}} \right)^t C_1.
    \label{eq:contHBproof}
\end{align}
In particular, for $T_0=n$, we obtain
\begin{align}
\label{eq:T0bound}
    \left\|
    \begin{bmatrix}
    \bx^{T_0} - \bx_* \\
    \bx^{T_0-1} - \bx_*
    \end{bmatrix}
\right\|_2
     & \leq \left(1-\frac{2\sqrt{\mu}}{\sqrt{L}+\sqrt{\mu}} \right)^{T_0} C_1 \\ & \leq \frac{1}{n} C_1 ,
\end{align}
as long as $n \geq \displaystyle  \frac{-\log(n)}{\log \left(1-\frac{2\sqrt{1/2}}{\sqrt{C \log (n)}+\sqrt{1/2}}\right)}$. Furthermore, the choice $t \geq n$ then implies that $\displaystyle \left(1-\frac{2 \sqrt{1/2}}{\sqrt{C \log(n)}+\sqrt{1/2}} \right)^t \leq \frac{1}{n}$. 
Therefore, as long as the iterates remain in the RIC for the first $T_0$ iterations, we get contraction to a neighborhood where $\dist(\bx^{T_0}, \bx_*) \leq C_1 / n = O(1/n)$. 

After this, for $t \geq T_0$, locality implies incoherence, i.e., condition \eqref{eq:loc} implies \eqref{eq:inc}. Indeed, for $t = T_0$, by Cauchy-Schwartz inequality, concentration inequality \eqref{eq:concnorm}, and inequality \eqref{eq:T0bound} we obtain:
\begin{align}
    &\max_{1 \leq \ell \leq m} |\ba_{\ell}^T(\bx^{T_0}-\bx_*)| \nonumber \\ & \leq \sqrt{6n} \|\bx^{T_0}-\bx_* \|\leq \sqrt{6n} \frac{1}{n} C_1 \nonumber \\& \lesssim C_2 \sqrt{\log(n)}. \label{eq:T0inco}
\end{align}
Now, because $\bx^{T_0}$ is in the RIC, we can use Lemma \ref{lem:contraction} to get that
\begin{align}
     \left\|
    \begin{bmatrix}
    \bx^{T_0+1} - \bx_* \\
    \bx^{T_0} - \bx_*
    \end{bmatrix}
\right\|_2 & \leq \left(1 - \frac{2 \sqrt{\mu}}{\sqrt{L}+\sqrt{\mu}} \right) \left\|
    \begin{bmatrix}
    \bx^{T_0} - \bx_* \\
    \bx^{T_0-1} - \bx_*
    \end{bmatrix}
\right\|_2 \nonumber \\
    &\leq \frac{1}{n} C_1 \label{T0plusbound}
\end{align}
Following the same argument that led to \eqref{eq:T0inco} and \eqref{T0plusbound}, we can see that locality implies incoherence for iterates with $t \geq T_0$.  We have just shown that the assumptions of Lemma \ref{lem:contraction} are satisfied for all $t \geq T_0=n$, which implies that
\begin{equation}
    \left\|
    \begin{bmatrix}
    \bx^{t+1} - \bx_* \\
    \bx^t - \bx_*
    \end{bmatrix}
\right\|_2 \leq \left( 1- \frac{2 \sqrt{\mu}}{\sqrt{L}+\sqrt{\mu}} \right)^t C_1
\end{equation}
\end{proof}

The previous proofs rely on the following three lemmas. We state the first without proof, since it was shown in \cite{ma2020implicit}. This lemma gives high probability bounds on the Hessian seen in \eqref{eq:MaLOCINC} provided that the point $\bx$ is in the RIC.
\begin{lem}[Lemma 1 of \citet{ma2020implicit}]\label{lem:scsm}
For sufficiently small $C_1$, sufficiently large $C_2$, and $m \geq c_0 n \log n$ with sufficiently large $c_0$, with probability at least $1-\cO(m n^{-10})$, one has
\begin{equation}
    \nabla^2 f (\bx) \preceq (5 C_2(10+C_2) \log (n)) \bI_n
\end{equation}
for all $\bx \in \mathbb{R}^n$ for which $\|\bx -\bx_* \|_2  \leq 2 C_1$, and
\begin{equation}
    \nabla^2 f(\bx) \succeq \frac{1}{2} \bI_n
\end{equation}
for all $\bx \in \mathbb{R}^n$ such that
\begin{align}
\label{eq:RIC}
    \|\bx -\bx_* \|_2 & \leq  2 C_1  \\ \nonumber
    \max_{i} |\ba_i^T(\bx-\bx_*)| &\leq C_2 \sqrt{\log (n)} \nonumber 
\end{align}
\end{lem}

In the following proof of Theorem \ref{thm:HB}, as a short hand, we use $\mu = 1/2$ and $L = c \log n$ as lower and upper bounds on the Hessian by Lemma \ref{lem:scsm}.

Next, using the previous lemma, we show that if the current pair of consecutive iterates lie within the RIC, then they get closer to $\bx_*$. 
\begin{lem}[Contraction]
\label{lem:contraction}
If $\ba_j \overset{i.i.d.}{\sim}N(\bzero, \bI)$,  then with probability at least $1-\mathcal{O}(mn^{-10})$, then the iterates \eqref{eq:HB}
\begin{equation}
\left\|
    \begin{bmatrix}
    \bx^{t+1} - \bx_* \\
    \bx^t - \bx_*
    \end{bmatrix}
\right\|_2    
    \leq \left( \frac{\sqrt{L}-\sqrt{\mu}}{\sqrt{L}+\sqrt{\mu}} \right)
    \left \|
    \begin{bmatrix}
        \bx^t-\bx_* \\
        \bx^{t-1}-\bx_*
    \end{bmatrix}
    \right\|_2
\end{equation}
when $\bx^t, \bx^{t-1}$ satisfy conditions \eqref{eq:RIC}, $\eta \lesssim \frac{4}{(\sqrt{1/2}+\sqrt{\log(n)})^2}$, and $\beta = \frac{\sqrt{c\log(n)}-\sqrt{1/2}}{\sqrt{c\log(n)}+\sqrt{1/2}}$.
\end{lem}

\begin{proof}[Proof of Lemma 2]
We can write one step of the heavy ball method as
\begin{align}
    \begin{bmatrix}
        \bx^{t+1} - \bx_* \\
        \bx^t - \bx_*
    \end{bmatrix}
    &=
    \begin{bmatrix}
        \bx^t - \bx_* -\eta \nabla f(\bx^t) + \beta (\bx^t-\bx^{t-1}) \\
        \bx^t- \bx_*
    \end{bmatrix} \nonumber \\ 
    &= 
    \bM(t)
    \begin{bmatrix}
        \bx^t - \bx_* \\
        \bx^{t-1} - \bx_*
    \end{bmatrix}
\end{align}
where $\boldsymbol{\xi}^t \in [\bx_*,\bx^t]$, the segment joining $\bx_*$ and $\bx^t$. Here we have made
\begin{equation}
 \label{eq:hbmatrix}
\bM(t)=
    \begin{bmatrix}
        (1+\beta) \bI - \nabla^2 f(\boldsymbol{\xi}^t) & -\beta \bI \\
        \bI & \boldsymbol{0}
    \end{bmatrix}
\end{equation}

A direct analysis shows that $\boldsymbol{\xi}^t$ is in the region of RIC.
Indeed, if $\boldsymbol{\xi}^t=(1-\tau) \bx_*+ \tau \bx^t$ for some $\tau \in (0,1)$ then $$\| \boldsymbol{\xi}^t - \bx_*\|_2 = \tau \| \bx^t - \bx_*\|_2 \leq 2 C_1$$ and $$\max_{i} |\ba_i^T(\boldsymbol{\xi}-\bx_*)| = \tau \max_{i}|\ba_i^t (\bx^t-\bx_*)| \leq C_2 \sqrt{\log n}.$$ Therefore, as $f$ is strongly convex and smooth, we can bound the norm of the matrix factor \eqref{eq:hbmatrix} as is usually done in the heavy ball method, by bounding its eigenvalues (see for instance \cite{polyak1964some,polyak1987introduction}).

Bounding the eigenvalues in this way yields the contraction
\begin{equation}
\left\|
    \begin{bmatrix}
    \bx^{t+1} - \bx_* \\
    \bx^t - \bx_*
    \end{bmatrix}
\right\|_2    
    \leq \left( \frac{\sqrt{L}-\sqrt{\mu}}{\sqrt{L}+\sqrt{\mu}} \right)
    \left \|
    \begin{bmatrix}
        \bx^t-x_* \\
        \bx^{t-1}-x_*
    \end{bmatrix}
    \right\|_2
\end{equation}
\end{proof}

We now put the previous two lemmas together to prove the main induction. This lemma ensures that pairs of consecutive iterates maintain the contraction and incoherence properties necessary to remain in the RIC. Note that the proof of this lemma relies on a further Lemma \ref{lem:loo}, where we show that the iterates \eqref{eq:HB} stay close to leave-one-out sequences.
\begin{lem}[Induction]
\label{lem:induction}
Suppose that we initialize \eqref{eq:HB} with $\bx^0=\sqrt{\frac{\lambda_1(\bY)}{3}} \widetilde{\bx}^0$, where $\lambda_1(\bY)$ and $\bx^0$ are the leading eigenvalue and eigenvector of $\bY=\frac{1}{m} \sum_{i=1}^{m} y_i \ba_i \ba_i^T $. Suppose further that $\ba_j \overset{i.i.d.}{\sim}N(\bzero, \bI)$. Then, for $n$ sufficiently large, with probability at least $1-\cO(mn^{-10})$,
\begin{align}\label{eq:induct_loc}
    \|\bx^{t+1}-\bx_* \|_2^2 + \|\bx^{t}-\bx_* \|_2^2 & \leq C_1^2 \\
    \max_{1 \leq i \leq m} |\ba_i^T(\bx^{t+1}-\bx_*)| &\leq C_2 \sqrt{\log(n)}
\end{align}
\end{lem}
\begin{proof}
The proof of all these relations relies on an induction argument. According to Lemma 5 in \cite{ma2020implicit},  for a given $\delta>0$ and $m$ large enough one has $\|\bx^0 - \bx_*\| \leq 2 \delta$. Choosing $\delta=\frac{C_1}{8}$ implies that $\|\bx^1-\bx_*\|^2+\|\bx^0-\bx_*\|^2 \leq C_1$. Furthermore, by Lemma \ref{lem:contraction}, for $\cC = \left( \frac{\sqrt{L}-\sqrt{\mu}}{\sqrt{L}+\sqrt{\mu}} \right)<1$,
    \begin{align}\label{eq:induct_locpf}
        &\|\bx^{t+1} -\bx_* \|_2^2 + \|\bx^t-\bx_* \|_2^2 \\ \nonumber
        & \leq \cC^2 (\|\bx^t-\bx_*\|_2^2+\|\bx^{t-1}-\bx_* \|_2^2) \leq C_1^2
    \end{align}

Then, we prove incoherence using induction as well:
\begin{align}
& \max_{1\leq \ell \leq m}|\ba_{\ell}^T (\bx^{t+1}-\bx_*)| \\
 \leq & |\ba_{\ell}^T(\bx^{t+1}-\bx^{t+1,(\ell)})+\ba_{\ell}^T(\bx^{t+1,(\ell)}-\bx_*)| \\
 \leq &\|\ba_{\ell}^T\| \| \bx^{t+1} -\bx^{t+1,(\ell)}\| + 5 \sqrt{\log(n)} \|\bx^{t+1,(\ell)}-\bx_*\| \\
\leq & \|\ba_{\ell}^T \| C_3 \sqrt{\frac{\log(n)}{n}} \\ 
&+ 5\sqrt{\log(n)} (\|\bx^{t+1,(\ell)}-\bx^{t+1} \| + \|\bx^{t+1}-\bx_*\|) \\
\leq & \sqrt{6n} C_3 \sqrt{\frac{\log(n)}{n}} + 5 \sqrt{\log(n)} \left(C_3 \sqrt{\frac{\log(n)}{n}} +C_1 \right) \\
\leq & C_2 \sqrt{\log(n)}
\end{align}
by taking $n$ large enough and using the fact that $C_2$ can be taken larger than $C_1$ and $C_3$. Here, the second to last inequality was obtained using \eqref{eq:induct_loc} and Lemma \ref{lem:loo}.

\end{proof}

To prove the upper bound on the incoherence in the previous Lemma, we must use leave-one-out sequences. Define the leave-one-out function
 $$f^{(\ell)}(\bx)=\frac{1}{4m}\sum_{i:i\neq \ell}((\ba_i^T \bx)^2-y_i)^2$$ 
 We let $\bx^{t, (\ell)}$ be sequence defined by running \eqref{eq:HB} on this function with the same initialization, $\bx^{0, (\ell)} = \bx^0$.

\begin{lem}[Leave one out]\label{lem:loo}

 Suppose that $\ba_j \overset{i.i.d.}{\sim}N(\bzero, \bI) $. Then, with probability at least $1-\cO(mn^{-10})$, for $n$ sufficiently large,
$$\max_{1 \leq \ell \leq m} \left\| \begin{bmatrix}
    \bx^{t+1}-\bx^{t+1,(\ell)} \\
    \bx^t -\bx^{t,(\ell)}
\end{bmatrix}\right\|_2\leq C_3 \sqrt{\frac{\log(n)}{n}}.$$
\end{lem}
\begin{proof} According to 6 in \cite{ma2020implicit} we have that
\begin{equation}
    \max_{1 \leq \ell \leq m} \| \bx^0-\bx^{0, (\ell)} \|_2 \leq C_3 \sqrt{\frac{\log (n)}{n}}
\end{equation}
for any constant $C_3$ (for $m$ large enough so that $\frac{n \log(n)}{m} \leq C_3$). Our base case is satisfied if $C_3$ is chosen appropriately and we take into account that $\bx^1=\bx^0, \bx^{1,(\ell)}=\bx^{0,(\ell)}$ (see Lemma 6 in \cite{ma2020implicit}).

Recall that $$\bx^{t+1} = \bx^t - \eta \nabla f (\bx^t) + \beta(\bx^t - \bx^{t-1}).$$ For the leave one out iteration we have:
\begin{align}
    \bx^{t+1}- \bx^{t+1, (\ell)} = \bx^t - \eta \nabla f (\bx^t) + \beta(\bx^t -\bx^{t-1})   \\
    -[\bx^{t,(\ell)}- \eta \nabla f^{(\ell)}(\bx^{t,\ell})] - \beta (\bx^{t, (\ell)} - \bx^{t-1, (\ell)}),
\end{align}
that can be rewritten as (by adding and subtracting the term $\eta \nabla f (\bx^{t,(\ell)})$):
\begin{align}
    \begin{bmatrix}
        \bx^{t+1}- \bx^{t+1, (\ell)} \\
        \bx^t -\bx^{t, (\ell)}
    \end{bmatrix}
 = &
\bM(t)
    \begin{bmatrix}
        \bx^t-\bx^{t,(\ell)} \\
        \bx^{t-1}-\bx^{t-1,(\ell)}
    \end{bmatrix} \\
    & - \eta \begin{bmatrix}
         \nabla f(\bx^{t,(\ell)}) -\nabla f^{(\ell)}(\bx^{t.(\ell)}) \\
         \boldsymbol{0}
    \end{bmatrix},
\end{align}
where \begin{equation}
\bM(t)=
    \begin{bmatrix}
     \bI - \eta \nabla^2 f(\boldsymbol{\xi}^t) + \beta \bI & -\beta \bI \\
     \bI & \boldsymbol{0}
    \end{bmatrix}
\end{equation}
for some $\boldsymbol{\xi}^t \in [\bx^t,\bx^{t,(\ell)}]$. As in the proof of Lemma \ref{lem:contraction}, $\boldsymbol{\xi}^t$  is in the RIC. If $\boldsymbol{\xi}^t= (1-\tau) \bx^t + \tau \bx^{t,(\ell)}$, then 
\begin{align}
\| \boldsymbol{\xi}^t-\bx_* \|_2 & \leq  (1-\tau) \| \boldsymbol{x}^t-\bx_* \|+\tau \|\bx^{t,(\ell)}-\bx_* \| \\
& \leq (1-\tau) C_1 + \tau (\|\bx^{t,(\ell)}-\bx^t\| + \|\bx^t-\bx_*\|) \\
& \leq (1-\tau) C_1+  \tau \left(C_3 \sqrt{\frac{\log(n)}{n}}+C_1\right) \\ 
& \leq 2 C_1
\end{align}
for large enough $n$. We can also show that $\boldsymbol{\xi}^t$ satisfies the incoherence condition using the induction hypothesis
\begin{align}
    &|\ba_{\ell}^T (\boldsymbol{\xi}^t - \boldsymbol{x}_*)| \\
    \leq & |(1-\tau)\ba_{\ell}^T(\bx^t-\bx_*) + \tau  \ba_{\ell}^T(\bx^{t,(\ell)}-\bx_*)| \\
\leq & (1-\tau) C_2 \sqrt{\log(n)} \\
& + \tau 5\sqrt{\log(n)} \| \bx^{t,(\ell)}-\bx^t + \bx^t -\bx_*\| \\
\leq & (1-\tau) C_2 \sqrt{\log(n)} \\
&+ \tau 5 \sqrt{\log(n)} \left( C_3 \sqrt{\frac{\log(n)}{n}} + C_1 \right) \\
\leq & C_2 \sqrt{\log(n)}
\end{align}

Thus, by bounding $\bM(t)$ and then using induction hypothesis, we obtain
\begin{align}
\label{eq:HBMbound}
\left\| \bM(t)
    \begin{bmatrix}
        \bx^t-\bx^{t,(\ell)} \\
        \bx^{t-1}-\bx^{t-1,(\ell)}
    \end{bmatrix} 
    \right\|_2  
    & \leq
    \left( \frac{\sqrt{L}-\sqrt{\mu}}{\sqrt{L}+\sqrt{\mu}} \right)
    \left\|
    \begin{bmatrix}
        \bx^t-\bx^{t,(\ell)} \\
        \bx^{t-1}-\bx^{t-1,(\ell)} 
    \end{bmatrix}
    \right\|_2  \\ \nonumber
    &\leq \left( \frac{\sqrt{L}-\sqrt{\mu}}{\sqrt{L}+\sqrt{\mu}} \right) C_3 \sqrt{\frac{\log(n)}{n}}.
\end{align}
On the other hand, by the definition of $f^{(\ell)}$ and \eqref{eq:1der},
\begin{align} \nabla f (x^{t,(\ell)})- \nabla f^{(\ell)} (\bx^{t,(\ell)})
& =  \frac{1}{m}[(\ba_{\ell}^T \bx^{t,(\ell)})^2-(\ba_{\ell}^T\bx_*)^2]\ba_{\ell} \ba_{\ell}^T \bx^{t,(\ell)}.
\end{align}
so that
\begin{align} \label{eq:graddiffbd}
\left\|
\begin{bmatrix}
         \nabla f(\bx^{t,(\ell)}) -\nabla f^{(\ell)}(\bx^{t.(\ell)}) \\
         \boldsymbol{0}
    \end{bmatrix}
\right\|_2  
&\leq \frac{1}{m} \|\ba_{\ell} \| |\ba_{\ell}^T \bx^{t,(\ell)}||(\ba_{\ell}^T \bx^{t,(\ell)})^2-(\ba_{\ell}^T \bx_*)^2|.
\end{align}

We proceed to bound each term in \eqref{eq:graddiffbd}. First,
\begin{equation}
    \| \ba_{\ell}\| \leq \sqrt{6 n}
\end{equation}
with probability at least $1-\cO(m e^{-1.5 n})$, by the concentration inequality \eqref{eq:concnorm}.

For the second factor, using the contraction part of Lemma \ref{lem:contraction} in \eqref{eq:induct_loc},
\begin{align}
    |\ba_{\ell}^T \bx^{t,(\ell)}| \leq |\ba_{\ell}^T||(\bx^{t,(\ell)}-\bx_*)| +|\ba_{\ell}^T \bx_*| & \leq 5 \sqrt{\log(n)} (\|\bx^{t,(\ell)}-\bx^t\| + \|\bx^t-\bx_* \|) +5 \sqrt{\log(n)} \\
    &\leq 5 \sqrt{\log(n)} \left(C_3 \sqrt{\frac{\log (n)}{n}} +C_1 \right) + 5 \sqrt{\log(n)} \\
    &\leq (C_4 +5) \sqrt{\log(n)}
\end{align}
for $n$ large enough so that $\sqrt{\frac{\log(n)}{n}}$ is as small as needed. We have also used concentration inequality \eqref{eq:concproj}.

For the third factor
\begin{align}
    |(\ba_{\ell}^T \bx^{t,(\ell)})^2-(\ba_{\ell}^T \bx_*)^2| & \leq |(\ba_{\ell}^T \bx^{t,(\ell)}-\ba_{\ell}^T \bx_*)^2 + 2 \ba_{\ell}^T(\bx^{t,(\ell)}-\bx_*) \ba_{\ell}^T \bx_*| \\
   & \leq |\ba_{\ell}^T (\bx^{t,(\ell)}-\bx_*)| |\ba_{\ell}^T(\bx^{t,(\ell)}-\bx_*)+2\ba_{\ell}^T \bx_*| \\
   & \leq C_2 \sqrt{\log(n)} (C_2 \sqrt{\log(n)}+ 10 \sqrt{\log(n)})
\end{align}
This follows from the fact that $\bx^{t, (\ell)}$ is independent of $\ba_\ell$, and by Gaussian concentration. The combination of all these bounds gives
\begin{align}
\label{eq:HBgradbound}
 \left\|
\begin{bmatrix}
         \nabla f(\bx^{t,(\ell)}) -\nabla f^{(\ell)}(\bx^{t.(\ell)}) \\
         \boldsymbol{0}
    \end{bmatrix}
\right\|_2 \nonumber 
& \leq \sqrt{6n} (C_4+5) \sqrt{\log(n)} C_2 (C_2+10) \log(n) \nonumber\\
& \leq C_4+5)C_4 (C_4+10) \frac{n \log (n)}{m} \sqrt{\frac{\log(n)}{n}} \\
&\leq c \sqrt{\frac{\log(n)}{n}}
\end{align}
for any $c$ provided $m$ is sufficiently large, $m \geq n \log(n)$.
Finally the combination of bounds \eqref{eq:HBMbound} and \eqref{eq:HBgradbound} yields
\begin{align}
\left\|
    \begin{bmatrix}
        \bx^{t+1}- \bx^{t+1, (\ell)} \\
        \bx^t -\bx^{t, (\ell)}
    \end{bmatrix}
\right\|_2 & \leq \left( \frac{\sqrt{L}-\sqrt{\mu}}{\sqrt{L}+\sqrt{\mu}}\right) C_3 \sqrt{\frac{\log(n)}{n}} + c \eta \sqrt{\frac{\log(n)}{n}},
\end{align}
and as $c$ can be made as small as desired and because of the relation between $\left(\frac{\sqrt{\kappa}-1}{\sqrt{\kappa}+1} \right)$ and $\eta$ then
\begin{equation}
 \left\|
    \begin{bmatrix}
        \bx^{t+1}- \bx^{t+1, (\ell)} \\
        \bx^t -\bx^{t, (\ell)}
    \end{bmatrix}
\right\|_2 \leq C_3 \sqrt{\frac{\log(n)}{n}}
\end{equation}
\end{proof}

\subsection{Proof of Theorem \ref{thm:HB} for Nesterov's Accelerated Gradient method}

\begin{proof}[Proof of Theorem \ref{thm:HB} for Nesterov's method \ref{eq:FG}]

Again, in the following, as a short hand, we use $\mu = 1/2$ and $L = c \log n$ as lower and upper bounds on the Hessian by Lemma \ref{lem:scsm}.

The proof follows the same pattern as for Polyak's method with the difference that the contraction constant for Nesterov's method is $\displaystyle \left(1-\frac{\sqrt{\mu}}{\sqrt{L}} \right)$ instead of $\displaystyle\left(1-\frac{2 \sqrt{\mu}}{\sqrt{L}+\sqrt{\mu}} \right)$ in equation \eqref{eq:contHBproof}.

For Nesterov's method it is also sufficient to choose $T_0=n$ because for $n$ large enough $$n \geq \displaystyle \frac{-\log(n)}{\log \left(1-\frac{\sqrt{\mu}}{\sqrt{L}}\right)} \geq \displaystyle  \frac{-\log(n)}{\log \left(1-\frac{\sqrt{1/2}}{\sqrt{C \log (n)}}\right)}$$ which implies that for $t \geq T_0 =n$
\begin{equation}
   \displaystyle \left(1-\frac{\sqrt{1/2}}{\sqrt{C \log(n)}} \right)^t \leq \frac{1}{n} 
\end{equation}
The rest of the proof stays the same as the proof of Theorem \ref{thm:HB} for \eqref{eq:HB} given in Appendix \ref{app:HBpf}.
\end{proof}

We provide the lemmas as before for completeness. Note that there are small differences in the proof and these lemmas in order to deal with the extrapolation in the gradient computation for \eqref{eq:FG}. In particular, it takes more work to ensure that the extrapolated point $\bx^t + \beta(\bx^t - \bx^{t-1})$ remains in the RIC. To start, however, we show that as long as $\bx^t, \bx^{t-1}$ points are sufficiently far inside the RIC, then the pair has a contraction property. 
\begin{lem}[Contraction for Nesterov]
\label{lem:contractionNest}
If $\ba_j \overset{i.i.d.}{\sim}N(\bzero, \bI)$,  then with probability at least $1-\mathcal{O}(mn^{-10})$, the iterates \eqref{eq:FG} satisfy
\begin{equation}
    \left\| \begin{bmatrix}
        \bx^{t+1}-\bx^* \\
        \bx^{t}-\bx^*
    \end{bmatrix}\right\|_2 = \left( 1-\frac{\sqrt{\mu}}{\sqrt{L}}\right) \left\|\begin{bmatrix}
        \bx^t-\bx^* \\
        \bx^{t-1}-\bx^*
    \end{bmatrix}\right\|_2
\end{equation}
when $\bx^t, \bx^{t-1}$ satisfies conditions \eqref{eq:RIC}, $\eta=\frac{1}{c \log n}$, and $\beta = \frac{\sqrt{c\log(n)}-\sqrt{1/2}}{\sqrt{c\log(n)}+\sqrt{1/2}}$.
\end{lem}
\begin{proof}

Let's look at the recursion for $\bx^t$ in Nesterov's method. In this case, we have
\begin{equation}
    \bx^{t+1} = \bx^t - \eta \nabla f(\bx^t + \beta(\bx^{t}-\bx^{t-1})) + \beta (\bx^t - \bx^{t-1})
\end{equation}
We then have
\begin{align*}
    \bx^{t+1} - \bx^* &= \bx^t -\bx^* \\ \nonumber
    &-\eta \nabla^2 f(\bxi^t)(\bx^t + \beta(\bx^t - \bx^{t-1}) - \bx^*) \\ 
    &-\beta(\bx^t - \bx^*) \\
    &+\beta(\bx^{t-1} - \bx^*)
\end{align*}
that can be rewritten as
\begin{equation}
    \begin{bmatrix}
        \bx^{t+1}-\bx^* \\
        \bx^{t}-\bx^*
    \end{bmatrix} = \bM(t) \begin{bmatrix}
        \bx^t-\bx^* \\
        \bx^{t-1}-\bx^*
    \end{bmatrix}
\end{equation}
where
\begin{equation}
    \bM(t) = \begin{bmatrix}
        (1+\beta)(\bI - \eta \nabla^2 f(\bxi^t)) & -\beta(\bI - \nabla^2 f(\bxi^t)) \\ \bI & \bzero
    \end{bmatrix}
\end{equation}
and $\bxi^t \in [\bx_*,\bx^t + \beta(\bx^t-\bx^{t-1})]$, the segment joining $\bx_*$ and $\bx^t + \beta(\bx^t-\bx^{t-1})$.

We can show directly that $\bxi^t$ is in the region of RIC. For some $\tau \in (0,1)$ we have
\begin{align}
    \| \bxi^t-\bx_*\|_2& = \| (1-\tau) \bx_* + \tau (\bx^t+\beta(\bx^t-\bx^{t-1}))-\bx_* \| \nonumber \\
    & = \tau \|\bx^t - \bx_* + \beta(\bx^t-\bx_*) + \beta( \bx_*-\bx^{t-1} )\| \nonumber \\
    & \leq \frac{C_1}{3} + 2 \beta \frac{C_1}{3} \leq 2 C_1 .
\end{align}
and for incoherence we get
\begin{align}
& |\ba_{\ell}^T (\bxi^t-\bx_*)| \\
  & \leq  |\ba_{\ell}^T ((1-\tau) \bx_* + \tau (\bx^t+\beta(\bx^t-\bx^{t-1}))-\bx_*)| \\
  & = \tau| \ba_{\ell}^T (\bx^t-\bx_* +\beta(\bx^t-\bx^{t-1}))| \\
  & \leq |(1+\beta) \ba_{\ell}^T (\bx^t-\bx_*) + \beta \ba_{\ell}^T (\bx_*-\bx^t)| \\
  & \leq (1+2 \beta) \frac{C_2}{3} \sqrt{\log(n)}\leq C_2 \sqrt{\log(n)}
\end{align}
Therefore, as $f$ can be considered strongly convex and smooth, we can bound the eigenvalues of the matrix $M(t)$ to obtain
\begin{equation}
    \left\| \begin{bmatrix}
        \bx^{t+1}-\bx^* \\
        \bx^{t}-\bx^*
    \end{bmatrix}\right\|_2 = \left( 1-\frac{\sqrt{\mu}}{\sqrt{L}}\right) \left\|\begin{bmatrix}
        \bx^t-\bx^* \\
        \bx^{t-1}-\bx^*
    \end{bmatrix}\right\|_2
\end{equation}
\end{proof}

\begin{lem}[induction for Nesterov]
\label{lem:inductionNest}
For $n$ sufficiently large, with probability at least $1-\cO(mn^{-10})$,
\begin{align}
    \|\bx^{t+1}-\bx_* \|_2^2 + \|\bx^{t}-\bx_* \|_2^2 & \leq (C_1/3)^2 \\
    \max_{1 \leq i \leq m} |\ba_i^T(\bx^{t+1}-\bx_*)| &\leq \frac{C_2}{3} \sqrt{\log(n)}
\end{align}
\begin{proof}
In exactly the same way as for heavy ball, Lemma 5 in \cite{ma2020implicit} implies that the conditions are satisfied for the base case. Now, by Lemma \ref{lem:contractionNest}, for $\cC =\left(1-\frac{\sqrt{\mu}}{\sqrt{L}} \right) < 1$,
    \begin{align}
        &\|\bx^{t+1} -\bx_* \|_2^2 + \|\bx^t-\bx_* \|_2^2 \\ \nonumber
        & \leq \cC^2 (\|\bx^t-\bx_*\|_2^2+\|\bx^{t-1}-\bx_* \|_2^2) \leq\left(\frac{C_1}{3} \right)^2
    \end{align}
We can also use induction to prove incoherence:
\begin{align}
& \max_{1\leq \ell \leq m}|\ba_{\ell}^T (\bx^{t+1}-\bx_*)| \\
 \leq & |\ba_{\ell}^T(\bx^{t+1}-\bx^{t+1,(\ell)})+\ba_{\ell}^T(\bx^{t+1,(\ell)}-\bx_*)| \\
 \leq &\|\ba_{\ell}^T\| \| \bx^{t+1} -\bx^{t+1,(\ell)}\| + 5 \sqrt{\log(n)} \|\bx^{t+1,(\ell)}-\bx_*\| \\
\leq & \|\ba_{\ell}^T \| C_3 \sqrt{\frac{\log(n)}{n}} \\ 
&+ 5\sqrt{\log(n)} (\|\bx^{t+1,(\ell)}-\bx^{t+1} \| + \|\bx^{t+1}-\bx_*\|) \\
\leq & \sqrt{6n} C_3 \sqrt{\frac{\log(n)}{n}} + 5 \sqrt{\log(n)} \left(C_3 \sqrt{\frac{\log(n)}{n}} +\frac{C_1}{3} \right) \\
\leq & \frac{C_2}{3} \sqrt{\log(n)}
\end{align}
by taking $n$ large enough, considering that $C_2$ can be chosen larger than $C_1$ and that $C_3$ can be chosen as small as necessary by taking a large enough $c_0$ in $m\geq c_0 n \log(n)$.
\end{proof}
\end{lem}
Here, as in the proof of the corresponding lemma for heavy ball, to prove the upper bound on the incoherence in the previous Lemma, we must use leave-one-out sequences. As before, the leave-one-out function is defined as
 $$f^{(\ell)}(\bx)=\frac{1}{4m}\sum_{i:i\neq \ell}((\ba_i^T \bx)^2-y_i)^2$$ 
 and we let $\bx^{t, (\ell)}$ be the sequence defined by running \eqref{eq:FG} on this function with the same initialization, $\bx^{0, (\ell)} = \bx^0$.
\begin{lem}[Leave one out for Nesterov]
For $n$ sufficiently large, with probability at least $1-\cO(mn^{-10})$,
    $$\max_{1 \leq \ell \leq m} \left\| \begin{bmatrix}
    \bx^{t+1}-\bx^{t+1,(\ell)} \\
    \bx^t -\bx^{t,(\ell)}
\end{bmatrix}\right\|_2\leq C_3 \sqrt{\frac{\log(n)}{n}}.$$
\end{lem}

\begin{proof}
In an analogous way to the heavy ball case, we can write one iteration of Nesterov acceleration as
\begin{align}
    &\begin{bmatrix}
        \bx^{t+1}-\bx^{t+1,(\ell)} \\
        \bx^{t}-\bx^{t,(\ell)}
    \end{bmatrix} = \bM_2(t) \begin{bmatrix}
        \bx^{t}-\bx^{t,(\ell)} \\
        \bx^{t-1}-\bx^{t-1,(\ell)}
    \end{bmatrix} \nonumber \\ \nonumber
    & - \eta \begin{bmatrix}
         \nabla f(\bx^{t,(\ell)} + \beta(\bx^{t,(\ell)} - \bx^{t-1,(\ell)}) ) \\
         \boldsymbol{0}
    \end{bmatrix} \\ 
    &+ \eta \begin{bmatrix}
          \nabla f^{(\ell)}(\bx^{t,(\ell)} + \beta(\bx^{t,(\ell)} - \bx^{t-1,(\ell)})) \\
         \boldsymbol{0}
    \end{bmatrix},
    \label{eq:looNineqdecomp}
\end{align}
where
\begin{equation}
    \bM_2(t) = \begin{bmatrix}
        (1+\beta)(\bI - \eta \nabla^2 f(\bxi^t)) & -\beta(\bI - \eta \nabla^2 f(\bxi^t)) \\ \bI & \bzero
    \end{bmatrix}.
\end{equation}
with $\bxi^t =(1-\tau) \bz^{t,(\ell)} + \tau \bz^t$, where $\tau \in (0,1)$ and
\begin{align}
    \bz^{t,(\ell)} &= \bx^{t,(\ell)} + \beta (\bx^{t,(\ell)}-\bx^{t-1,(\ell)}) \\
    \bz^t & = \bx^t + \beta(\bx^t-\bx^{t-1}) .
\end{align}
 We can show directly that $\bxi^t$ is in the region of RIC. First, we have that
 \begin{align}
     \| \bxi^t-\bx_* \| &= (1-\tau) \|\bz^{t,(\ell)} - \bx_* \| + \tau \| \bz^t -\bx_* \| \\
     &\leq (1-\tau) (\|\bz^{t,(\ell)}- \bz^t \| +         \| \bz^t- \bx_* \|) + \tau \| \bz^t -\bx_* \| \\
     &\leq (1-\tau) \|\bz^{t,(\ell)} -\bz^t \| + \| \bz^t -\bx_*\|.
 \end{align}
 We can bound $\|\bxi^t - \bx_* \|$ in two parts:
 \begin{align}
     \| \bz^t - \bx_*\| & \leq (1+\beta) \|\bx^t - \bx_* \| + \beta \|\bx^{t-1}-\bx_*\| \\
     & \leq (1 + 2 \beta)\frac{C_1}{3}
 \end{align}
 and
 \begin{align}
      \|\bz^t - \bz^{t,(\ell)} \|
     & \leq (1 + \beta) \|\bx^t -\bx^{t,(\ell)} \| + \beta \|\bx^{t-1}-\bx^{t-1,(\ell)} \| \\
     & \leq (1+ 2\beta) C_3 \sqrt{\frac{\log{n}}{n}}
 \end{align}
so that
\begin{align}
    \| \bz^t-\bz^{t,(\ell)}\| & \leq (1-\tau) (1+ 2\beta) C_3 \sqrt{\frac{\log{n}}{n}} +  (1 + 2 \beta) \frac{C_1}{3}  \\
    & \leq 2 C_1
\end{align}
 by taking $n$ large enough.

Now for the incoherence condition we have
\begin{align}
     |\ba_{\ell}^T ((1-\tau) \bz^{t,(\ell)}-\bz^t)| 
     \leq &(1-\tau) 5 \sqrt{\log(n)} \|\bz^{t,(\ell)}- \bx_* \| \\
     &+ \tau |\ba_{\ell}^T (\bx^t-\bx_* + \beta(\bx^t-\bx_* +\bx_*-\bx^{t-1}))| \\
     \leq & (1-\tau) 5 \sqrt{\log(n)} \|\bz^{t,(\ell)} -\bz^t +\bz^t-\bx_* \| \\
     &+ \tau |\ba_{\ell}^T ((1+\beta) (\bx^t-\bx_*)+(\bx_*-\bx^{t-1}))| \\
     \leq & (1-\tau) 5 \sqrt{\log(n)} (\|\bx^{t,(\ell)}-\bx^t \|)  \\
     &+(1-\tau) 5 \sqrt{\log(n)} (\beta \|(\bx^{t,(\ell)}-\bx^t +\bx^{t-1}-\bx^{t-1,(\ell)} \|) \\
     & + (1-\tau) 5 \sqrt{\log(n)} (1+2\beta) C_1 \\
     & + \tau ((1+\beta) C_2 \sqrt{\log(n)}+ C_2 \sqrt{\log(n)}) \\
     \leq & (1-\tau) 5 \sqrt{\log(n)} \left((1+2 \beta) C_3 \sqrt{\frac{\log(n)}{n}} +(1+2\beta) C_1 \right) \\
     &+ \tau(2+\beta) C_2 \sqrt{\log(n)} \\
     \leq & \frac{C_2}{3} \sqrt{\log(n)} 
\end{align}
for large enough $n$.

As $\bxi^t$ is in the RIC, for the choice $\eta=\frac{1}{L}$, $\beta=\frac{\sqrt{L}-\sqrt{\mu}}{\sqrt{L}+\sqrt{\mu}}$ the eigenvalues of $\bM_2(\cdot)$ are bounded by $1-\sqrt{\mu}/\sqrt{L}$. After using our induction hypothesis, we get
\begin{align}
\left\| \bM_2(t) \right\|_2 \left\| \begin{bmatrix}
        \bx^{t}-\bx^{t,(\ell)} \\
        \bx^{t-1}-\bx^{t-1,(\ell)}
    \end{bmatrix} \right\|_2  &\leq \left(1-\frac{\sqrt{\mu}}{\sqrt{L}} \right) \left\| \begin{bmatrix}
        \bx^{t}-\bx^{t,(\ell)} \\
        \bx^{t-1}-\bx^{t-1,(\ell)}
    \end{bmatrix} \right\|_2  \\
    & \leq  \left(1-\frac{\sqrt{\mu}}{\sqrt{L}} \right) C_3 \sqrt{\frac{\log (n)}{n}} \label{eq:looNineq1}
\end{align}
On the other hand, if we make $\bz^{t,(\ell)} = \bx^{t,(\ell)}+\beta(\bx^{t,(\ell)}-\bx^{t-1,(\ell)})$, by the definition of $f^{(\ell)}$ we have:
\begin{align}
    \nabla f (\bz^{t,(\ell)}) - \nabla f^{(\ell)} (\bz^{t,(\ell)}) \\ = \frac{1}{m} ((\ba_{\ell}^T \bz^{t,(\ell)})^2-(\ba_{\ell}^T \bx_*)^2) \ba_{\ell} \ba_{\ell}^T \bz^{t,(\ell)}  \\
\end{align}
so that
\begin{align}
    \left\|\begin{bmatrix}
        \nabla f (\bz^{t,(\ell)}) - \nabla f^{(\ell)}  (\bz^{t,(\ell)}) \\
        \boldsymbol{0}
    \end{bmatrix} \right\|_2 \\
\leq \frac{1}{m} \|\ba_{\ell}\| |\ba_{\ell}^T \bz^{t,(\ell)} | |(\ba_{\ell}^T \bz^{t,(\ell)})^2-(\ba_{\ell}^T \bx_*)^2| \label{eq:Nestegraddiffbd}
\end{align}
\end{proof}

We now can bound each term in \eqref{eq:Nestegraddiffbd}. We have that
\begin{equation}
    \| \ba_{\ell}\| \leq \sqrt{6n}
\end{equation}
with probability at least $1-\cO(m e^{-1.5 n})$, by concentration inequality \eqref{eq:concnorm}. For the second factor we have
\begin{align}
    |\ba_{\ell}^T \bz^{t,(\ell)}| \leq & |\ba_{\ell}(\bx^{t,(\ell)}-\bx^t)| + |\ba_{\ell} (\bx^t -\bx_*)| \nonumber \\
    & +\beta|\ba_{\ell} (\bx^{t,(\ell)}-\bx^t)| + \beta |\ba_{\ell}^T (\bx^t-\bx_*)| \nonumber \\
    & + \beta |\ba_{\ell}(\bx_*-\bx^{t-1})| + \beta |\ba_{\ell}^T(\bx^{t-1}-\bx^{t-1,(\ell)})| \nonumber \\
    & + |\ba_{\ell}^T \bx_*|
\end{align}
Using the induction hypothesis and locality we obtain
\begin{align}
    |\ba_{\ell}^T \bz^{t,(\ell)}| \leq & 5 \sqrt{\log(n)} \left( C_3 \sqrt{\frac{\log(n)}{n}} + C_1 \right. \\& \left. + C_3 \sqrt{\frac{\log(n)}{n}} + 2 \beta C_1 + \beta C_3 \sqrt{\frac{\log(n)}{n}}  \right) \\
    & + 5 \sqrt{\log(n)}
\end{align}
so that
\begin{align}
    |\ba_{\ell}^T \bz^{t,(\ell)}|  \leq & 5 \sqrt{\log(n)} \left( (1+2\beta) C_3 \sqrt{\frac{\log(n)}{n}} +(1+2 \beta) C_1 \right) \\
   & + 5 \sqrt{\log(n)} \\
 \leq & (C_4 + 5) \sqrt{\log(n)} 
\end{align}
for sufficiently large $n$ so that $\sqrt{\frac{\log(n)}{n}}$ is as small as needed.
For the third factor, by making $\bz^{t,(\ell)} = \bx^{t,(\ell)}+\beta(\bx^{t,(\ell)}-\bx^{t-1,(\ell)})$,
\begin{align}
    |(\ba_{\ell}^T \bz^{t,(\ell)})^2-(\ba_{\ell}^T\bx_*)^2| 
    \leq &|(\ba_{\ell}^T \bz^{t,(\ell)}-\ba_{\ell}^T\bx_*)^2 - 2 \ba_{\ell}^T (\bz^{t,(\ell)}-\bx_*) \ba_{\ell}^T \bx_* | \\
    \leq &|\ba^T(\bz^{t,(\ell)}-\bx_*)| |\ba_{\ell}^T (\bz^{t,(\ell)}-\bx_*) -2 \ba_{\ell}^T\bx_*| \nonumber \\
    \leq &(C_2 \sqrt{\log(n)}+ 2 \beta C_2 \sqrt{\log(n)}) \\
    & \cdot (C_2 \sqrt{\log(n)}+2 \beta C_2 \sqrt{\log(n)}+10 \sqrt{\log(n)}) \\
    = & C_2 (1+2 \beta) \sqrt{\log(n)} ((1+2 \beta) C_2 \sqrt{\log(n)}+10 \sqrt{\log(n)})
\end{align}
Combining all these bounds gives
\begin{align}
    \frac{1}{m} \|\nabla f (\bz^{t, (\ell)})-\nabla f^{(\ell)} (\bz^{t,(\ell)})\| 
    & \leq \frac{1}{m} \sqrt{6n} (C_4+5) \sqrt{\log(n)} C_2(C_2+10) \log(n) \\
    & \leq (C_4 +5)C_2 (C_2+10) \frac{n \log(n)}{m} \sqrt{\frac{\log(n)}{n}} \\
    & \leq c \sqrt{\frac{\log(n)}{n}} \label{eq:looNineq2}
\end{align}
where $c$ can be made as small as needed by taking $m$ large. Finally, the combination of bounds \eqref{eq:looNineq1} and \eqref{eq:looNineq2} with equation \eqref{eq:looNineqdecomp} yields
\begin{align}
    \begin{bmatrix}
        \bx^{t+1} -\bx^{t+1,(\ell)} \\
        \bx^t -\bx^{t,(\ell)}
    \end{bmatrix}
    &\leq  \left[\left(1-\frac{\sqrt{\mu}}{\sqrt{L}}\right) +c \eta \right] C_3 \sqrt{\frac{\log(n)}{n}}  \\ &\leq C_3 \sqrt{\frac{\log(n)}{n}}
\end{align}
because $c$ can be made as small as needed and $\eta=\frac{1}{L}$.

\end{document}

%% file: short_commands.tex
\newcommand{\R}{\mathbb R}

\newcommand{\C}{\mathbb C}
\newcommand{\N}{\mathbb N}
\newcommand{\E}{\mathbb E}

\def\bI{\boldsymbol{I}}

\def\bM{\boldsymbol{M}}

\def\bS{\boldsymbol{S}}

\def\bU{\boldsymbol{U}}

\def\bY{\boldsymbol{Y}}

\def\cA{\mathcal{A}}

\def\cC{\mathcal{C}}

\def\cN{\mathcal{N}}
\def\cO{\mathcal{O}}

\def\ba{\boldsymbol{a}}

\def\bx{\boldsymbol{x}}
\def\by{\boldsymbol{y}}
\def\bz{\boldsymbol{z}}
\def\bzero{\boldsymbol{0}}

\def\bxi{\boldsymbol{\xi}}

\DeclareMathOperator{\dist}{\mathrm{dist}}